\documentclass[draft,nonblindrev]{informs1} %

\usepackage{natbib}
 \NatBibNumeric
 \def\bibfont{\small}%
 \bibpunct[, ]{[}{]}{,}{n}{}{,}%

\usepackage[colorlinks=true,breaklinks=true,bookmarks=true,urlcolor=blue,
     citecolor=blue,linkcolor=blue,bookmarksopen=false,draft=false]{hyperref}

\usepackage[utf8]{inputenc} %
\usepackage[T1]{fontenc}    %
\usepackage{url}            %
\usepackage{booktabs}       %
\usepackage{amsfonts}       %
\usepackage{nicefrac}       %

\usepackage{algorithmic, algorithm}

\usepackage{amssymb,amsmath}
\usepackage{multirow,multicol}

\usepackage{xcolor,xspace}
\usepackage{lscape}
\usepackage{subfigure,graphicx,epsfig,tikz,caption}
\usepackage[normalem]{ulem}
\usepackage{enumerate}
\usepackage{verbatim}
\usepackage{makeidx,latexsym}

\usepackage{bm}
\usepackage{lscape}
\usepackage[figuresright]{rotating}

\usepackage[capitalize]{cleveref}
\crefrangeformat{equation}{(#3#1#4)--(#5#2#6)}
\crefformat{equation}{(#2#1#3)}

\newcommand{\grad}{\nabla}
\newcommand{\la}{\langle}
\newcommand{\ra}{\rangle}
\newcommand{\bbR}{\mathbb{R}}
\newcommand{\bbE}{\mathbb{E}}
\newcommand{\bbP}{\mathbb{P}}

\DeclareMathOperator{\Prox}{Prox}

\DeclareMathOperator{\OPT}{OPT}
\newcommand{\cX}{\mathcal{X}}

\newcommand{\cP}{\mathcal{P}}
\newcommand{\cD}{\mathcal{D}}
\newcommand{\cB}{\mathcal{B}}
\newcommand{\cT}{\mathcal{T}}
\newcommand{\cM}{\mathcal{M}}

\newcommand{\vx}{\bm{x}}

\def\EMAIL#1{\href{mailto:#1}{#1}}%

\TheoremsNumberedThrough     %

\EquationsNumberedThrough    %

\begin{document}

\RUNAUTHOR{Lee et al.}

\RUNTITLE{Non-Smooth, H\"older-Smooth, and Robust Submodular Maximization}

\TITLE{Non-Smooth, H\"older-Smooth, and Robust Submodular Maximization}

\ARTICLEAUTHORS{%
\AUTHOR{Duksang Lee}
\AFF{Department of Industrial and Systems Engineering, KAIST, Daejeon, South Korea, \EMAIL{duksang@kaist.ac.kr}}
\AUTHOR{Nam Ho-Nguyen}
\AFF{Discipline of Business Analytics, The University of Sydney, Sydney, Australia, \EMAIL{nam.ho-nguyen@sydney.edu.au}}
\AUTHOR{Dabeen Lee}
\AFF{Department of Industrial and Systems Engineering, KAIST, Daejeon, South Korea, \EMAIL{dabeenl@kaist.ac.kr}}
} %

\ABSTRACT{%
We study the problem of maximizing a continuous DR-submodular function that is not necessarily smooth. We prove that the continuous greedy algorithm achieves an $[(1-1/e)\OPT-\epsilon]$ guarantee when the function is monotone and H\"older-smooth, meaning that it admits a H\"older-continuous gradient. For functions that are non-differentiable or non-smooth, we propose a variant of the mirror-prox algorithm that attains an $[(1/2)\OPT-\epsilon]$ guarantee. We apply our algorithmic frameworks to robust submodular maximization and distributionally robust submodular maximization under Wasserstein ambiguity. In particular, the mirror-prox method applies to robust submodular maximization to obtain a single feasible solution whose value is at least $(1/2)\OPT-\epsilon$. For distributionally robust maximization under Wasserstein ambiguity, we deduce and work over a submodular-convex maximin reformulation whose objective function is H\"older-smooth, for which we may apply both the continuous greedy and the mirror-prox algorithms.
}%

\MSCCLASS{
68W40, %
68W25, %
90C15, %
90C17, %
90C25, %
90C26, %
90C27, %
}
\ORMSCLASS{Primary: ; secondary: }
\HISTORY{}

\maketitle

\section{Introduction.}

Submodularity is a structural property that is exploited in both discrete and continuous domains and has numerous applications in optimization.
Submodular set functions are often associated with the problem of selecting the optimal group of items from a large pool of discrete candidates. Optimizing submodular set functions has applications in sensor placement~\cite{sensor-placement}, feature and variable selection~\cite{variable-selection}, dictionary learning~\cite{DK11}, document summarization~\cite{HLT10,HLT11}, image summarization~\cite{image1,streaming-matchoid2,streaming-matchoid3}, and active set selection in non-parametric learning~\cite{active-set}. Continuous submodularity also naturally arises in a wide range of application domains where the objective exhibits the diminishing returns property, such as MAP inference for determinantal point processes and mean-field inference of probabilistic graphical models~\cite{bian-long}, influence maximization with marketing strategies~\cite{influence} and robust budget allocation~\cite{soma-budget,staib-robust}.

An important line of research regarding submodularity is robust submodular optimization. In many application scenarios, submodular objective functions are often non-stationary or have underlying randomness, e.g., sensor placement~\cite{Krause-robust} and influence maximization~\cite{robust-influence-1,robust-influence-2,robust-influence-3}. Such functions are typically estimated by data and thus subject to estimation noise. The robust optimization framework considers a family of submodular functions and attempts to maximize them simultaneously by taking the point-wise minimum of the functions~\cite{Krause-robust,chen-robust,equator,Torrico-robust-2}. %
The distributionally robust submodular maximization framework extends robust submodular maximization. It considers a set of distributions over a family of submodular functions and attempts to maximize the minimum expectation over the distributions~\cite{dro-sfm}.

The celebrated continuous greedy algorithm~\cite{matroidcon,vondrak-dr} is used for maximizing a submodular set function over a matroid constraint and is based on optimizing the multilinear extension of the given submodular set function. A modification of the continuous greedy algorithm extends to solve continuous DR-submodular maximization~\cite{bian}, and there are other first-order methods for the continuous case~\cite{grad-sfm,conditional}. The multilinear extension of a submodular set function is differentiable and smooth, meaning that it has a Lipschitz continuous gradient, and the existing algorithms for continuous submodular maximization require smoothness as well. However, taking the point-wise minimum for robust maximization does not preserve differentiability nor smoothness, and similarly, the minimum of the expectation of a family of submodular functions is non-differentiable and non-smooth in general. For this reason, the aforementioned methods cannot directly apply to robust submodular maximization, and the existing methods for the robust problem cannot avoid violating constraints~\cite{Krause-robust,chen-robust,equator,torrico-robust-1,Torrico-robust-2}. The existing methods for distributionally robust problem require an expensive smoothing step~\cite{equator} or a high sample variance assumption~\cite{dro-sfm} to guarantee smoothness.

\paragraph{Our contributions.} Motivated by this, we study the problem of maximizing monotone up-concave continuous functions, which all robust versions of submodular functions satisfy. Our methods work for functions which do not admit a Lipschitz continuous gradient, and may not even be differentiable.
\begin{itemize}
\item  We prove that the continuous greedy algorithm works for a continuous up-concave function as long as it is differentiable and has a H\"older continuous gradient, i.e., H\"{o}lder-smooth. Here, H\"older continuity is a relaxed notion of Lipschitz continuity. We show that the algorithm returns a solution whose value is at least $(1-1/e)\OPT-\epsilon$ after $O(1/\epsilon^{\frac{1}{\sigma}})$ iterations, where $\OPT$ is the optimal value and $\sigma$ is the H\"older exponent.

\item For functions that are non-differentiable or differentiable but non-smooth, we propose a variant of the mirror-prox method that achieves a value at least $\OPT/2-\epsilon$ after $O(1/\epsilon^{\frac{2}{1+\sigma}})$ iterations. In particular, we have $\sigma=0$ for non-differentiable functions, in which case the number of required iterations is $O(1/\epsilon^2)$.

\item We apply our methods to robust DR-submodular maximization problems, and show that the mirror-prox method converges to a solution that achieves value at least $\OPT/2-\epsilon$ after $O(1/\epsilon^{2})$ iterations. We show how this can be applied to discrete multiobjective submodular problems \citep{swap-rounding,Udwani2021}, improving on existing guarantees. We also consider a distributionally robust formulation of DR-submodular maximization under Wasserstein ambiguity, and show that it can be cast as a robust DR-submodular maximization problem. We apply a standard Nesterov smoothing technique and provide approximation bounds under smoothing, but show that the resulting function is non-smooth, but admits a H\"older continuous gradient with parameter $\sigma = 1/2$. We then show that, for the distributionally robust problem, the continuous greedy algorithm returns a solution with value at least $(1-1/e)\OPT-\epsilon$ after $O(1/\epsilon^2)$ iterations, while the mirror-prox method achieves a value at least $\OPT/2-\epsilon$ after $O(1/\epsilon^{\frac{4}{3}})$ iterations.

\end{itemize}

\section{Related work.}

Starting with the work of~\citet{NemWolseyFisher1978}, algorithms for maximizing a submodular set function under various types of constraints have been extensively studied. One popular example is the continuous greedy algorithm~\cite{vondrak-dr,matroidcon} that was first introduced to solve the problem under a matroid constraint, and the algorithm applies to other complicated constraints such as a system of linear constraints and the intersection of multiple matroid constraints~\cite{matroidcon,kulik}. Its main idea is to solve the continuous relaxation obtained by the so-called multilinear extension of a given submodular set function and then apply rounding procedures to the solution from the continuous relaxation. Since then, the multilinear extension has gained attention~\cite{vondrak-dr,matroidcon,swap-rounding,kulik, sim-annealing,mor-curvature}. The multilinear extension of a submodular set function has nice properties such as smoothness and DR-submodularity, which generalizes the diminishing returns property of set functions to continuous functions. 

Algorithms for maximizing differentiable and smooth continuous DR-submodular functions were developed in several works. \citet{bian} proposed a variant of the continuous greedy algorithm for the problem of maximizing a continuous DR-submodular function. The algorithm utilizes a linear oracle over the constraint set, which is assumed to be down-closed and convex, and it returns a solution whose value is at least $(1-1/e)\text{OPT}-\epsilon$ after $O(d/\epsilon)$ iterations where $d$ is the ambient dimension. \citet{grad-sfm} showed that the stochastic gradient ascent achieves a value at least $\text{OPT}/2-\epsilon$ after $O(1/\epsilon)$ iterations under an arbitrary convex constraint set. Moreover, they considered a restricted setting where exact gradients are unavailable, but their unbiased stochastic estimates with a bounded variance are available. They showed that the same asymptotic guarantee can be achieved after $O(1/\epsilon^2)$ iterations. \citet{conditional} proposed the stochastic continuous greedy algorithm for the stochastic case that guarantees a value at least $(1-1/e)\text{OPT}-\epsilon$ after $O(1/\epsilon^3)$ iterations. \citet{scg++} and \citet{one-sample} then developed variants of the stochastic continuous greedy algorithm that attain the same guarantee after $O(1/\epsilon^2)$ iterations.

A related area is robust submodular maximization, where the aim is to maximize the point-wise minimum of multiple submodular functions~\cite{Krause-robust}. One of its major application domains is robust influence maximization~\cite{robust-influence-1,robust-influence-2,robust-influence-3}, where the goal is to maximize the worst-case influence spread characterized by multiple information diffusion processes. \citet{Krause-robust} showed that there is no polynomial time constant approximation algorithm for robust submodular maximization unless 
$P=NP$. They considered the cardinality constrained case, for which they developed a bi-criteria approximation algorithm which sacrifices feasibility to achieve an approximation guarantee. \citet{chen-robust,equator,Torrico-robust-2,torrico-robust-1} generalized the idea of bi-criteria approximation for general constraints such as matroid and knapsack constraints. Unfortunately, the continuous greedy method and other gradient based methods~\cite{bian,grad-sfm,conditional,scg++,one-sample} do not directly apply to the robust problem, as the robust submodular objective is not smooth in general. The mirror-prox algorithm was applied to smooth submodular problems by \citet{AdibiEtAl2022}. Our work expands its applicability to H\"{o}lder-smooth and non-smooth submodular problems by extending the results of \citet{Nemirovski2004,StonyakinEtAl2022}.

Distributionally robust submodular maximization aims to maximize the minimum expected function value where the minimum is taken over a family of probability distributions on a given set of submodular objective functions. Therefore, it is a generalization of both robust and stochastic submodular maximization. Its formulation depends on how we construct the family of distributions, and a popular way is to consider distributions that are close to a reference distribution, typically set to the empirical distribution of some past observations, based on a metric measuring the distance between two probability distributions. \citet{dro-sfm} considered the distributionally robust submodular maximization framework defined by the $\chi^2$-divergence. They showed that under a high sample variance condition, the distributionally robust objective becomes smooth. However, such a smooth assumption does not hold in general, thus \citet{equator} proposed applying a randomized smoothing technique of~\citet{duchi-smoothing}.

An ambiguity set based on the $\chi^2$-divergence contains only distributions whose support is the same as that of the reference distribution. In particular, when we use the empirical distribution on a finite sample as a reference, we cannot consider any data points outside the sample set. Using a Wasserstein ambiguity set alleviates this since it considers optimal transport of data points, subject to a transportation budget. \citet{MohajerinEsfahaniKuhn2018,BlanchetMurthy2019} showed that various distributionally robust problems with Wasserstein ambiguity can be reformulated using semi-infinite duality theory. In the context of DR-submodular maximization, a reformulation of the problem with Wasserstein ambiguity leads to a H\"{o}lder-smooth objective. Our work analyzing the continuous greedy and mirror-prox algorithms show that they can be applied to this application.

\section{Problem setting.}\label{sec:prelim}

In this paper, we will study the following problem:
\begin{equation}\label{SFM}
	\sup_{x\in\mathcal{X}} F(x)\tag{SFM-c}
\end{equation}
where $\mathcal{X} \subseteq \bbR_+^d$ is a convex constraint set and $F:\bbR^d \to \bbR$ satisfies assumptions in \cref{def:up-concave} stated below.
Before providing the assumptions, we discuss the motivation behind them. The notion of \emph{diminishing returns} is captured for a continuous function $F$ through the concept of \emph{DR-submodularity}: 
$$F(x+z e_i)-F(x)\geq F(y+z e_i)-F(y)$$
for any
$x,y\in \mathcal{X}$ 
such that $x\leq y$, standard basis vector $e_i\in\mathbb{R}^d$ and $z\in\mathbb{R}_+$. This generalizes the notion of submodularity for set functions, and, indeed, certain continuous extensions of submodular set functions (such as the multilinear extension and the softmax extension) are DR-submodular. There are many other classes of DR-submodular functions~\cite{bach}. A quadratic function $F(x)=x^\top A x/2 + b^\top x+c$ is DR-submodular if and only if all entries of $A$ are nonpositive. Functions involving $\varphi_{ij}(x_i-x_j)$ where $\varphi_{ij}:\mathbb{R}\to\mathbb{R}$ for $i,j\in[m]$ are convex, $g(\sum_{i\in [m]}\lambda_i x_i)$ where $g$ is concave and $\lambda\geq 0$, and $\log\det(\sum_{i\in[m]}x_i A_i)$ where matrices $A_i$'s are positive definite and $x\geq 0$ are all examples of DR-submodular functions. In this paper, however, we do not just consider DR-submodular objectives, but instead consider structured generalizations of them.

\subsection{Motivation: robust DR-submodular function maximization.}\label{sec:prelim-motivation} Consider a family of DR-submodular functions $f(x, p)$ where $p$ is a vector belonging to a given parameter set $\cP$. We are interested in the \emph{robust DR-submodular maximization} problem
\begin{equation}\label{eq:robust-SFM-c}
	\max_{x \in \cX} F(x), \quad F(x) := \min_{p \in \cP} f(x, p).\tag{R-SFM-c}
\end{equation}
Such an objective has arisen, for example, in submodular equilibrium computation \cite{equator} and its distributionally robust version with $\chi^2$-ambiguity \cite{dro-sfm}. Specifically, both \citet{equator,dro-sfm} consider finitely many functions $f_i(x)$ for $i \in [n]$, each one arising from the multilinear extension of a submodular set function (hence is DR-submodular), and seek to solve
\[ \max_{x \in \cX} F(x), \quad F(x) := \min_{p \in \cP} \sum_{i \in [n]} p_i f_i(x). \]
\citet{equator} consider $\cP = \Delta_n$ the $n$-simplex, whereas \citet{dro-sfm} consider $\cP = \left\{ p \in \cP : \frac{1}{2} \|n p - \bm{1}\|_2^2 \leq \rho \right\}$ for some fixed $\rho \in (0,(n-1)/2)$. Note that both \citet{equator,dro-sfm} are motivated by discrete objectives $\tilde{f}_i(S)$ where $S$ are subsets of some ground set of items $V$, and aim to choose a \emph{distribution} of sets $S \in \cD$ so that the minimum expected objective $\min_{p \in \cP} \sum_{i \in [n]} p_i \bbE_{S \sim \cD}[\tilde{f}_i(S)]$ is maximized. Using a correlation gap result, \citet{equator} shows that this problem can be approximately solved by considering $f_i(x)$ as the multilinear extension of $\tilde{f}_i(S)$ up to factor $1-1/e$. Therefore, they focus on the continuous objective $\min_{p \in \cP} \sum_{i \in [n]} p_i f_i(x)$. Importantly, both \cite{equator,dro-sfm} rely on smooth structure implied by the multilinear extensions, or employ smoothing techniques to obtain smoothness, thus enabling the application of conditional gradient (Frank-Wolfe) type algorithms.

In this paper, we focus on developing algorithms for more general types of robust objectives. One important example is \emph{distributionally robust DR-submodular maximization} over broader classes of functions. More specifically, consider a possibly infinite class of DR-submodular functions $\tilde{f}(x,\xi)$ parametrized by a vector $\xi \in \Xi$ where $\Xi$ belongs to a Euclidean space. The vector $\xi$ is assumed to be randomly drawn from $\Xi$ according to some distribution $Q$, which leads us to the objective $\bbE_{\xi \sim Q}[\tilde{f}(x,\xi)]$. However, when there is some uncertainty on $Q$, we consider a family of distributions $\mathcal{B}$ over $\Xi$ containing $Q$ and the following distributionally robust optimization problem
\begin{equation}\label{eq:DR-SFM}
	\max_{x \in\cX} F(x), \quad \text{where} \quad F(x) :=  \inf_{P\in\mathcal{B}}\mathbb{E}_{\xi\sim P}\left[\tilde{f}(x, \xi) \right].\tag{DR-SFM-c}
\end{equation}
For $\cB$, we consider the \emph{2-Wasserstein ball} of radius $\theta$, based on the $\ell_2$-norm, centered at distribution $Q$:
$$\cB := \mathcal{B}_2(Q,\theta)=\left\{P:\ d_{W,2}(P,Q)\leq \theta\right\}$$
where $d_W(P,Q)$ denotes the \emph{2-Wasserstein distance} between two probability distributions $P,Q$ over sample space $\Xi$, which is defined as
$$d_{W,2}(P,Q)=\inf_{\Pi}\left\{\left(\mathbb{E}_{(\xi,\zeta)\sim\Pi}\left[\|\xi-\zeta\|_2^2\right]\right)^{1/2}:\ \text{$\Pi$ has marginal distributions $P, Q$} \right\}.$$
In fact, we show that this into the framework of \eqref{eq:robust-SFM-c} via well-known reformulation techniques.
\begin{proposition}\label{prop:reformulation-1}
	Suppose $Q=\sum_{i\in[n]}q_i\delta_{\xi^i}$ for some $q \in \Delta_n$, $\xi^i \in \Xi$ and that $q_i>0$ for each $i\in[n]$. Furthermore, suppose that $\tilde{f}(x,\xi)$ is convex in $\xi$ for each $x$. Then for every $x \in \cX$, we have
	\begin{equation}\label{DR-SFM-G-2}
		\inf_{P \in \mathcal{B}_2(Q,\theta)} \mathbb{E}_{\xi \sim P} \left[ f(x,\xi) \right] = \inf_{p \in \cP} f(x,p),
	\end{equation}
	where
	\begin{subequations}\label{eq:DR-SFM-wasserstein-form}
	\begin{align}
	\cP &:= \left\{(\zeta^1,\ldots,\zeta^n): \sum_{i \in [n]} q_i \|\xi^i - \zeta^i\|_2^2 \leq \theta^2,\ \zeta^i\in\Xi\ \ \forall i\in[n] \right\}, \quad p := (\zeta^1,\ldots,\zeta^n)\label{eq:DR-SFM-wasserstein-uncertainty}\\
	f(x, p) &:= \sum_{i\in[n]} q_i \tilde{f}(x,\zeta^i).\label{eq:DR-SFM-wasserstein-objective}
	\end{align}
	\end{subequations}
\end{proposition}
The proof is in Appendix \ref{app:proofs}. Now, we are interested in solving \eqref{eq:robust-SFM-c} with $F(x)$ defined via \eqref{eq:DR-SFM-wasserstein-form}. The issue is that this $F$ is no longer guaranteed to be DR-submodular, despite being constructed from DR-submodular functions $\tilde{f}(x,xi)$, and also does not enjoy smoothness properties which previous DR-submodular function maximization techniques are reliant on. Fortunately, $F$ still satisfies a property called \emph{up-concavity} (defined below) which is critical for convergence of continuous DR-submodular maximization algorithms. As mentioned above, \citet{equator} employ a smoothing technique to obtain smoothness. However, due to the special structure of $F$, we will show in \cref{sec:dro} that application of a standard smoothing technique results in a \emph{H\"{o}lder-smooth} function (defined below), rather than a smooth one.

Therefore, this necessitates the development of new non-smooth and H\"{o}lder-smooth algorithms for continuous up-concave maximization, which is the motivation for our paper.

\subsection{Definitions and structural assumptions.}\label{sec:prelim-structural}
With this motivation in mind, we consider the following structural assumptions on our objective $F$.
\begin{definition}\label{def:up-concave}
	A function $F: \bbR^d \to \bbR$ is:
	\begin{itemize}
		\item \emph{up-concave} (equivalently \emph{coordinate-wise concave}~\cite{bian,equator}) if for any fixed $x\in\bbR^d$ and non-negative direction $v \in \bbR_+^d$, the function $t \mapsto F(x + tv)$ is concave.
		
		\item \emph{monotone} if $F(y) \geq F(x)$ whenever $y \geq x$ and $x, y \in \bbR^d$.
		
		\item \emph{non-negative} if $F(x) \geq 0$ for all $x \in \bbR^d$.
	\end{itemize}
\end{definition}
\citet[Proposition 2]{bian} show that all DR-submodular functions are up-concave, but also that up-concave maximization \eqref{SFM} exactly may still be hard \cite[Proposition 3]{bian}, thus we look for approximate solutions in the following sense.
\begin{definition}\label{def:approx-solution}
	We say that a solution $\bar{x}\in\mathcal{X}$ is an $(\alpha, \epsilon)$-approximate solution to~\eqref{SFM} for some $0\leq \alpha<1$ and $\epsilon>0$ if $$F(\bar{x})\geq \alpha\cdot \sup_{x\in \mathcal{X}} F(x) - \epsilon.$$
\end{definition}
In this paper, we develop first-order iterative algorithms that guarantee an $(\alpha, \epsilon)$-approximate solution to~\eqref{SFM} where the number of required iterations is bounded by $O(1/\epsilon^c)$ for some constant $c$. Before presenting these, we describe further structural assumptions on $F$ that we will use.

The existing literature focuses on solving \eqref{SFM} when the objective function $F$ is differentiable and smooth, that is, there exists some $\beta>0$ such that
$$\|\grad F(x)-\grad F(y)\|_*\leq \beta \|x-y\|$$
for any $x,y\in\mathbb{R}^d$. However, the up-concave functions that arise from robust and distributionally robust DR-submodular maximization problems are not necessarily smooth.
Therefore, in this paper, we consider the case when $F$ is non-smooth and potentially non-differentiable as well.
Firstly, when $F$ is non-differentiable, it is known that first-order information for convex functions can still be captured via subdifferentials and subgradients, and analogously for superdifferentials/gradients for concave functions. However, up-concave functions $F$ may be neither convex nor concave. Therefore, to capture first-order information for $F$, we introduce the notion of up-super-gradients defined as follows.
\begin{definition}\label{def:up-supergradient}
	We say that $g_x$ is an \emph{up-super-gradient} at $x\in\bbR^d$ if 
	\begin{equation}\label{eq:up-supergradient}
		F(y)\leq F(x)+g_x^\top(y-x)\quad\text{for any $y\in\mathbb{R}^d$ such that $y\geq x$ or $y\leq x$.}
	\end{equation}
	We define $\partial^\uparrow F(x)$ as the set of up-super-gradients of $F$ at $x$.
\end{definition}
As expected, when $F$ is differentiable, $\partial^\uparrow F(x)$ reduces to the gradient.
\begin{lemma}\label{diff-differential}
	If $F:\mathbb{R}^d \to \mathbb{R}$ is differentiable and up-concave, $\partial^\uparrow F(x) = \left\{\grad F(x)\right\}$ for any $x\in\mathbb{R}^d$.
\end{lemma}
The proof of Lemma~\ref{diff-differential} is given in Appendix~\ref{app:proofs}.

We relax the notion of smoothness via the notion of \emph{H\"older-smoothness}, defined as follows. 
\begin{definition}\label{def:holder-smooth}
We say that a function $F:\mathbb{R}^d\to\mathbb{R}$ is \emph{H\"older-smooth} with respect to a norm $\|\cdot\|$ and a monotone increasing function $h:\mathbb{R}_+\to\mathbb{R}_+$ with $\lim_{x\to0} h(x)=0$ if for any $x,y\in\mathbb{R}^d$, there exist $g_x\in\partial^\uparrow F(x)$ and $g_y\in\partial^\uparrow F(y)$ such that
\begin{equation}
	\|g_x-g_y\|_*\leq h(\|x-y\|).\label{eq:holder-smooth-nd-1}
\end{equation}
When $F$ is differentiable, this condition becomes
\begin{equation}\label{eq:holder-smooth}
\|\grad F(x)-\grad F(y)\|_*\leq h(\|x-y\|).
\end{equation}
\end{definition}
Here, if $h$ is given by $h(z)=\beta\cdot z$ for some constant $\beta>0$, then H\"older-smoothness with respect to $h$ reduces to smoothness. Throughout the paper, the particular form of $h$ we focus on is
\begin{equation}\label{eq:holder-smooth-h}
h(z)=\sum_{j \in [k]} \beta_j \cdot z^{\sigma_j}
\end{equation}
for some integer $k\geq 1$ and constants $\beta_1,\ldots,\beta_k>0$ and $0<\sigma_1,\ldots,\sigma_k\leq 1$. Although we allow $h$ to have multiple terms, i.e., $k>1$, a more common definition of H\"older smoothness is by a function $h$ with a single term~\cite{stone-smooth}. If there is no function $h$ of the form~\eqref{eq:holder-smooth-h} where at least one of $\sigma_1,\ldots,\sigma_k$ being positive with which $F$ satisfies condition~\eqref{eq:holder-smooth}, then we say that $F$ is \emph{non-smooth}.

\section{Continuous greedy method.}\label{sec:cg}

In this section, we consider~\eqref{SFM}, the problem of maximizing a monotone up-concave function, for the differentiable case. We present Algorithm~\ref{alg:conditional-gradient-upconcave-monotone}, which is a variant of the continuous greedy algorithm ~\cite{bian} for monotone up-concave function maximization. 
\begin{algorithm}%
	\caption{Continuous greedy algorithm for maximizing a differentiable monotone up-concave function.}\label{alg:conditional-gradient-upconcave-monotone}
	\begin{algorithmic}
		\STATE{Initialize $x_0\leftarrow 0$.}
		\FOR{$t=1,\ldots,T$}
		\STATE{Fetch $g_t$.}
		\STATE{$v_t=\argmax_{v\in\mathcal{X}} g_t^\top v$.}
		\STATE{Update $x_t = \left( 1 - \frac{1}{t} \right)x_{t-1} + \frac{1}{t} v_t = \frac{1}{t} (v_1 + \cdots + v_t)$.}
		\ENDFOR
		\RETURN{$x_T = \frac{1}{T} (v_1 + \ldots v_T)$.}
	\end{algorithmic}
\end{algorithm}

In Algorithm~\ref{alg:conditional-gradient-upconcave-monotone}, we take $x_t=(v_1+\cdots+v_t)/t$ at each iteration $t$, in which case, $x_t$ is a convex combination of points in $\mathcal{X}$. It is more common to take $x_t = (v_1+\cdots+v_t)/T$ for the continuous greedy algorithm~\cite{bian}, but the point is not necessarily a point in $\mathcal{X}$ when $t<T$. Since $F$ is monotone, scaling up $(v_1+\cdots+v_t)/T$ to obtain $(v_1+\cdots+v_t)/t$ always improves the objective, so we take the latter.

The following is the main lemma to relate the iterates of \cref{alg:conditional-gradient-upconcave-monotone} to the function $F$. Importantly, we allow errors in computing the gradient at each iteration; specifically, the $\kappa_t$ terms capture the effect of error in the gradient terms $g_t$ through the gaps $\|\grad F((t-1)x_{t-1}/T)-g_t\|_*$. If each of these are $\leq \delta$, then the total accumulated error due to inexact gradients is $O(\delta)$. Furthermore, the results given by~\cite{bian} assume that the objective function $F$ is smooth, but in our analysis, we show that H\"older-smoothness is sufficient to guarantee a finite convergence to an approximate solution.
\begin{lemma}\label{lemma:cg-complexity}
Let $F:\mathbb{R}^d\to\mathbb{R}_+$ be a differentiable, monotone, and up-concave function. Assume that $\sup_{x\in\mathcal{X}}\|x\|\leq R$ for some $R>0$ and there exists a function $\tilde h: \mathbb{R}_+\to\mathbb{R}_+$ such that for any $x,y \in \mathbb{R}^d$,
\begin{equation}\label{eq:H-smooth-generalized-0}
	F(x) \geq F(y) + \grad F(y)^\top (x-y) - \tilde h(\|x-y\|).
\end{equation}
Let $x_0,v_1,x_1,\ldots,v_T,x_T$ be the sequence generated by Algorithm \ref{alg:conditional-gradient-upconcave-monotone}.  Then
\begin{equation}\label{eq:fw-main-step} 
F(x_T) \geq \left( 1 - \frac{1}{e}\right) \sup_{x\in\mathcal{X}}F(x) - \sum_{s=0}^{T-1} \left( 1 - \frac{1}{T} \right)^s \kappa_{T-s} + \frac{1}{e}F(0)
\end{equation}
where for $1\leq t\leq T-1$, 
$$\kappa_t = \frac{2R}{T} \left\|\grad F\left(\frac{t-1}{T}x_{t-1}\right) - g_t\right\|_* + \tilde h\left(\left\|\frac{1}{T}v_t\right\|\right).$$
\end{lemma}
\begin{proof}{Proof of Lemma~\ref{lemma:cg-complexity}.}
Take $y_t=tx_t/T = (v_1+\ldots+v_t)/T$ for $0\leq t\leq T$ and $x^*\in\argmax_{x\in\mathcal{X}}F(x)$.
Then for each $1\leq t\leq T$,
\begin{align}\label{fw:eq1}
	\begin{aligned}
		F(y_t) 
		&\geq F(y_{t-1}) + \grad F(y_{t-1})^\top (y_t-y_{t-1}) - \tilde h(\|y_t - y_{t-1}\|)\\
		&= F(y_{t-1})+\frac{1}{T} \grad F(y_{t-1})^\top v_t - \tilde h \left(\left\|\frac{v_t}{T}\right\|\right).
	\end{aligned}
\end{align}
Next we define $u:\mathbb{R}_+\to\mathbb{R}_+$ as $u(\gamma)=F(y_{t-1} + \gamma x^*)$ for $\gamma\geq0$. %
Since $F$ is up-concave, differentiable and $x^* \in \cX \subset \bbR_+^d$, $u$ is concave and differentiable. In particular, we have $F(y_{t-1} + \gamma x^*) - F(y_{t-1})=u(\gamma) - u(0) \leq u'(0)\gamma = \gamma \grad F(y_{t-1})^\top x^*$. Consequently, setting $\gamma = 1$ results in
\begin{equation}\label{fw:eq2}
	F(y_{t-1} + x^*) - F(y_{t-1}) \leq \grad F(y_{t-1})^\top x^*. 
\end{equation}
Furthermore, $F(y_{t-1} + x^*) \geq F(x^*)$ since $F$ is monotone and $y_{t-1} + x^* \geq x^*$. Then it follows that
\begin{align}\label{fw:eq3}
	\begin{aligned}
		F(x^*) - F(y_{t-1}) &\leq F(y_{t-1}+x^*)-F(y_{t-1})\\
		&\leq \grad F(y_{t-1})^\top x^*\\
		&= g_t^\top x^* + (\grad F(y_{t-1}) - g_t)^\top x^*\\
		&\leq g_t^\top v_t + (\grad F(y_{t-1}) - g_t)^\top x^*\\
		&= \grad F(y_{t-1})^\top v_t + (\grad F(y_{t-1}) - g_t)^\top(x^* - v_t)
	\end{aligned}
\end{align}
where the second inequality is from~\eqref{fw:eq2} and the third inequality is due to our choice of $v_t$. Combining~\eqref{fw:eq1} and~\eqref{fw:eq3}, we have
\begin{align*}
	F(y_t) &\geq F(y_{t-1})+\frac{1}{T} \grad F(y_{t-1})^\top v_t - \tilde h\left(\left\|\frac{v_t}{T}\right\|\right)\\
	&\geq F(y_{t-1}) + \frac{1}{T} \left( F(x^*) - F(y_{t-1}) \right) - \frac{1}{T} (\grad F(y_{t-1}) - g_t)^\top (x^* - v_t) - \tilde h \left(\left\|\frac{v_t}{T}\right\|\right).
\end{align*}
Define $\eta_t = F(x^*) - F(y_{t-1})$. Then 
\begin{align}\label{eq:cg-step1}
	\begin{aligned}
		\eta_t - \eta_{t+1}&=F(y_t)-F(y_{t-1}) \\
		&\geq \frac{1}{T} \eta_t - \frac{1}{T} (\grad F(y_{t-1}) - g_t)^\top (x^* - v_t) - \tilde h\left(\left\|\frac{v_t}{T}\right\|\right)\\
		&\geq \frac{1}{T} \eta_t - \frac{1}{T} \| \grad F(y_{t-1}) - g_t\|_*\| x^* - v_t\| - \tilde h\left(\left\|\frac{v_t}{T}\right\|\right)\\ 
		&\geq \frac{1}{T} \eta_t - \frac{2 R}{T}\| \grad F(y_{t-1}) - g_t\|_*- \tilde h\left(\left\|\frac{v_t}{T}\right\|\right)
	\end{aligned}
\end{align}
where the second inequality is by H\"older's inequality and the third inequality is from our assumption that $\sup_{x\in\mathcal{X}}\|x\|\leq R$. Then~\eqref{eq:cg-step1} implies that $\eta_{t+1}\leq \left(1-{1}/{T}\right)\eta_{t}+\kappa_{t}$.
Unwinding this recursion, we get
\[ \eta_{t+1} \leq \left( 1 - \frac{1}{T} \right)^{t} \eta_1 + \sum_{s=0}^{t-1} \left( 1 -\frac{1}{T} \right)^s \kappa_{t-s}. \]
In particular, when $t=T$, we obtain
\[ F(x^*) - F(y_T) \leq \left( 1 - \frac{1}{T} \right)^{T} \left( F(x^*) - F(y_0) \right) + \sum_{s=0}^{T-1} \left( 1 - \frac{1}{T} \right)^s \kappa_{T-s}.\]
Since $y_0=0$, $y_T=x_T$, and $(1-1/T)^T\leq 1/e$, it follows that 
\[ \left( 1 - \frac{1}{e} \right) (F(x^*) - F(0)) \leq F(x_T) - F(0) + \sum_{s=0}^{T-1} \left( 1 - \frac{1}{T} \right)^s \kappa_{T-s} \]
as required.
\Halmos
\end{proof}

We now show that if $F$ is H\"older-smooth with respect to a norm $\|\cdot\|$ and a function $h:\mathbb{R}_+\to\mathbb{R}_+$ of the form~\eqref{eq:holder-smooth-h},
then there exists a function $\tilde h:\mathbb{R}_+\to\mathbb{R}_+$ satisfying~\eqref{eq:H-smooth-generalized-0}.

\begin{lemma}\label{lemma:holder-smooth}
If a differentiable function $F:\mathbb{R}^d\to\mathbb{R}_+$ is H\"older-smooth with respect to a norm $\|\cdot\|$ and $h$, then for any $x,y\in\mathbb{R}^d$,
\begin{equation}\label{eq:H-smooth-generalized}
	F(x) \geq F(y) + \grad F(y)^\top (x-y) - \sum_{j \in [k]} \frac{\beta_j}{1+\sigma_j}\|x-y\|^{1+\sigma_j}.
\end{equation}
\end{lemma}
\begin{proof}{Proof.}
Let $x,y\in\mathbb{R}^d$. Since $F$ is differentiable, we have that
\begin{equation}\label{eq:cg-iteration-1}
F(x)- F(y) = \int_{0}^1 \grad F(y+t(x-y))^\top (x-y)dt.
\end{equation}
Next we show that the following relations hold.
\begin{align}\label{eq:cg-iteration-2}
\begin{aligned}
&\left|F(x)- F(y)-\grad F(y)^\top(x-y)\right|\\&= \left| \int_{0}^1 (\grad F(y+t(x-y))-\grad F(y))^\top (x-y) dt\right|\\
&\leq \int_{0}^1 \left| (\grad F(y+t(x-y))-\grad F(y))^\top (x-y) \right|dt\\
&\leq \int_{0}^1 \|\grad F(y+t(x-y))-\grad F(y)\|_*\cdot\|x-y\|dt\\
&\leq \int_{0}^1 h(t\|x-y\|)\cdot\|x-y\|dt\\
&=\int_{0}^1 \left(\sum_{j \in [k]} \beta_j t^{\sigma_j}\|x-y\|^{1+\sigma_j}\right)dt\\
&=\sum_{j \in [k]} \frac{\beta_j}{1+\sigma_j}\|x-y\|^{1+\sigma_j}
\end{aligned}
\end{align}
where the first equality comes from~\eqref{eq:cg-iteration-1}, the second inequality is by H\"older's inequality. Then~\eqref{eq:H-smooth-generalized} follows from~\eqref{eq:cg-iteration-2}.
\Halmos
\end{proof}

Based on Lemmas~\ref{lemma:cg-complexity} and~\ref{lemma:holder-smooth}, we obtain the following convergence result on Algorithm~\ref{alg:conditional-gradient-upconcave-monotone}.
\begin{theorem}\label{thm:conditional-gradient}
Let $F:\mathbb{R}^d\to\mathbb{R}_+$ be a differentiable, monotone, and up-concave function that is H\"older-smooth with respect to  a norm $\|\cdot\|$ and $h$. Assume that for each iteration of Algorithm~\ref{alg:conditional-gradient-upconcave-monotone}, $g_t$ is chosen so that $\|\grad F((t-1)x_{t-1}/T) - g_t\|_* \leq \delta$ and $\sup_{x\in\mathcal{X}}\|x\|\leq R$ for some $R>0$. Then 
\begin{equation}\label{eq:cg-performance-bound}
F(x_T) - F(0) \geq \left( 1 - \frac{1}{e} \right)\sup_{x\in\mathcal{X}} \left\{ F(x) - F(0) \right\} - \left( 2\delta R + \sum_{j \in [k]}\frac{\beta_j R^{1+\sigma_j}}{T^{\sigma_j}} \right).
\end{equation}
\end{theorem}
\begin{proof}{Proof.}
For $1\leq t\leq T$, we have
\begin{equation*}
\kappa_t = \frac{2R}{T} \left\| \grad F\left( \frac{t-1}{T} x_{t-1} \right) - g_t \right\|_* + \sum_{j \in [k]} \frac{\beta_j}{1+\sigma_j}\|v_t/T\|^{1+\sigma_j} \leq \frac{2\delta R}{T} + \sum_{j \in [k]}\frac{\beta_j}{1+\sigma_j}\left(\frac{R}{T}\right)^{1+\sigma_j}
\end{equation*}
because $\| \grad F((t-1)x_{t-1}/T) - g_t\|_*\leq \delta$ and $\|v_t\|\leq R$.
Therefore, it follows that 
\begin{align*}
\sum_{s=0}^{T-1} \left( 1 - \frac{1}{T} \right)^s \kappa_{T-s} &\leq \frac{1}{T} \sum_{s=0}^{T-1} \left( 1 - \frac{1}{T} \right)^s \left( 2\delta R + \sum_{j \in [k]}\beta_j \left(\frac{R^{1+\sigma_j}}{T^{\sigma_j}}\right)\right)\\
&= \left( 1 - \left(1 - \frac{1}{T}\right)^T \right) \left( 2\delta R + \sum_{j \in [k]}\beta_j \left(\frac{R^{1+\sigma_j}}{T^{\sigma_j}}\right) \right)\\
&\leq \left( 2\delta R + \sum_{j \in [k]}\beta_j \left(\frac{R^{1+\sigma_j}}{T^{\sigma_j}}\right) \right).
\end{align*}
Plugging this to~\eqref{eq:fw-main-step} in Lemma~\ref{lemma:cg-complexity}, we deduce that
\[ \left( 1 - \frac{1}{e} \right) (F(x^*) - F(x_0)) \leq F(x_T) - F(x_0) +  \left( 2\delta R + \sum_{j \in [k]}\beta_j \left(\frac{R^{1+\sigma_j}}{T^{\sigma_j}}\right) \right) \]
as required.
\Halmos
\end{proof}

As an immediate corollary, we obtain the following convergence result.
\begin{corollary}\label{cor:cg-convergence-rate}
Let $F:\mathbb{R}^d\to\mathbb{R}_+$ be a differentiable, monotone, and up-concave function that is H\"older-smooth with respect to a norm $\|\cdot\|$ and $h$. Let $\sigma:=\min_{j\in[k]}\sigma_j$. If $F(0) = 0$ and $\delta = O(\epsilon)$, then Algorithm~\ref{alg:conditional-gradient-upconcave-monotone} returns an $(1-1/e,\epsilon)$-approximate solution after $O(1/\epsilon^{1/\sigma})$ iterations.
\end{corollary}
By Corollary~\ref{cor:cg-convergence-rate}, as long as $\min_{j\in[k]}\sigma_j>0$, Algorithm~\ref{alg:conditional-gradient-upconcave-monotone} converges to an $(1-1/e,\epsilon)$ after a finite number of iterations. However, we have $\min_{j\in[k]}\sigma_j=0$ for the non-smooth case, in which case, the performance bound~\eqref{eq:cg-performance-bound} given in Theorem~\ref{thm:conditional-gradient} does not guarantee that Algorithm~\ref{alg:conditional-gradient-upconcave-monotone} converges to an $(1-1/e,\epsilon)$-approximate solution.

\section{Mirror-prox method.}\label{sec:mp}

For the continuous greedy algorithm, given by Algorithm~\ref{alg:conditional-gradient-upconcave-monotone}, to guarantee finite convergence, we need the objective function $F$ to be differentiable and H\"older-smooth. If $F$ is non-differentiable or non-smooth, then the continuous greedy algorithm does not necessarily converge to a desired approximate solution. Motivated by this, we develop an algorithm for \eqref{SFM} that guarantees finite convergence to an approximate solution even when $F$ is non-differentiable or non-smooth. 

We propose a variant of the mirror-prox method, given in Algorithm~\ref{mirror-prox-non-differentiable}, which achieves a finite convergence to a constant approximation. In particular, Algorithm~\ref{mirror-prox-non-differentiable} admits the case when $\min_{j\in [k]}\sigma_j=0$ even if Algorithm~\ref{alg:conditional-gradient-upconcave-monotone} only works only when $\min_{j\in[k]}\sigma_j>0$. Moreover, Algorithm~\ref{mirror-prox-non-differentiable} does not require $F$ to be differentiable unlike Algorithm~\ref{alg:conditional-gradient-upconcave-monotone}.

At the heart of our analysis is the following lemma, which allows us to achieve a constant factor approximation guarantee. It is similar in spirit to the first-order characterization of concave functions. \citet{grad-sfm} have already considered the differentiable case, while our lemma extends the result to the non-differentiable case.
\begin{lemma}\label{lemma:hassani_subgrad}
	If $F$ is nonnegative, monotone and up-concave, then for any $x,y\in\mathbb{R}^d$ and $g\in\partial^\uparrow(F(x))$, we have
	$$F(y)-2F(x)\leq g^\top(y-x).$$
\end{lemma}
\begin{proof}{Proof.}
	Let $x,y\in\mathcal{X}$. Since $x\wedge y\leq x\leq x\vee y$, it follows from \cref{def:up-supergradient} that we have
	\begin{align*}
		F(x\wedge y)-F(x)&\leq g^\top(x\wedge y-x)\\
		F(x\vee y)-F(x)&\leq g^\top(x\vee y-x)
	\end{align*}
	for every $g\in\partial^\uparrow(F(x))$. Adding up the two inequalities, we obtain
	\[
	-2F(x)+F(x\vee y)+F(x\wedge y)\leq g^\top(x\vee y+x\wedge y-2x)=g^\top(y-x).
	\]
	Since $F$ is nonnegative and monotone, it follows that $F(y)-2F(x)\leq F(x\vee y)-2F(x)\leq g^\top (y-x)$.
	\Halmos
\end{proof}
The general strategy is as follows. Suppose we have a sequence of chosen points $\{x_t : t \in \cT\}$ with up-super-gradients $g(x_t) \in \partial^\uparrow F(x_t)$ for each $t \in \cT$. Then by \cref{lemma:hassani_subgrad}, for any $y \in \cX$ and each $t$ we have $F(y) - 2F(x_t) \leq g(x_t)^\top (y-x_t)$. Multiplying this by nonnegative weights $\gamma_t \geq 0$ and summing, we get
\begin{align*}
\left( \sum_{t \in \cT} \gamma_t \right) F(y) &\leq 2 \sum_{t \in \cT} \gamma_t F(x_t) + \sum_{t \in \cT} \gamma_t g(x_t)^\top (y-x_t)\\
\frac{1}{2} F(y) &\leq \left( \sum_{t \in \cT} \gamma_t \right)^{-1} \sum_{t \in \cT} \gamma_t F(x_t) + \frac{1}{2} \left( \sum_{t \in \cT} \gamma_t \right)^{-1} \sum_{t \in \cT} \gamma_t g(x_t)^\top (y-x_t)\\
&\leq \max_{t \in \cT} F(x_t) + \frac{1}{2} \left( \sum_{t \in \cT} \gamma_t \right)^{-1} \sum_{t \in \cT} \gamma_t g(x_t)^\top (y-x_t),
\end{align*}
where the second inequality holds since the maximum of $F(x_t)$ over $t \in \cT$ is larger than any convex combination. Since this holds for every $y$, we can take the maximum of the right hand side over $y \in \cX$, then the same for the left hand side, so that we have
\begin{subequations}\label{eq:mirror-prox-basic-bound}
	\begin{align}
	\frac{1}{2} \max_{x \in \cX} F(x) &\leq \max_{t \in \cT} F(x_t) + \epsilon_{\cT}\\
	\text{where } \epsilon_\cT &:= \frac{1}{2} \max_{y \in \cX} \left( \sum_{t \in \cT} \gamma_t \right)^{-1} \sum_{t \in \cT} \gamma_t g(x_t)^\top (y-x_t).
	\end{align}
\end{subequations}
This means that we have found a $(1/2,\epsilon_\cT)$-approximate solution for \cref{SFM} (\cref{def:approx-solution}). Of course, we need to ensure $\epsilon_\cT$ is sufficiently small, which we will do by choosing the points $x_t$ via the mirror prox method.

Let $\Phi$ be a mirror map that is 1-strongly convex on $\mathbb{R}^d$ with respect to the norm $\|\cdot\|$. We define
\begin{equation*}
	V_{x}(z) := \Phi(z) - \Phi(x) - \la \grad \Phi(x), z-x \ra
\end{equation*}
which is often referred to as the \emph{Bregman divergence} associated with $\Phi$. Then $V_{x}(z) \geq \frac{1}{2} \|x-z\|^2$. Let $\Prox$ be a proximal operator given by
\begin{equation*}
	\Prox_{x}(\xi) := \argmin_{z \in \mathcal{X}} \left\{ \la \xi, z \ra + V_{x}(z) \right\}.
\end{equation*}
The mirror-prox method is presented in \cref{mirror-prox-non-differentiable}.
\begin{algorithm}%
	\caption{Mirror-prox algorithm for~\eqref{SFM}.}\label{mirror-prox-non-differentiable}
	\begin{algorithmic}
		\STATE{Initialize $v_1\in \argmin_{x\in\mathcal{X}}\Phi(x)$.}
		
		\FOR{$t=1,\ldots,T-1$}
		\STATE{Fetch $\tilde{g}_t$.}
		
		\STATE{$x_t=\operatorname{Prox}_{v_t}(-\gamma_t \tilde{g}_t)$.}
		
		\medskip
		\STATE{Fetch $g_t$.}
		
		\STATE{$v_{t+1}=\operatorname{Prox}_{v_t}(-\gamma_t g_t)$.}
		\ENDFOR
		\RETURN{%
			$x^*\in\argmax \left\{F(x):x\in\{x_t:\ t \in \cT\}\right\}$ where $\cT = \{\lfloor (T-2)/3 \rfloor +1,\ldots,T-1\}$.}
	\end{algorithmic}
\end{algorithm}

At each iteration $t$, we obtain an estimate $\tilde{g}_t$ of an up-super-gradient of $F$ at $v_t$ and an estimate $g_t$ of an up-super-gradient of $F$ at $x_t$. As long as the estimation error, measured by the norm of the gap, is bounded by $\delta$ at each iteration $t$, the overall error will be $O(\delta)$.

Our candidate points will be $\{x_t : t \in \cT\}$ generated from \cref{mirror-prox-non-differentiable}. Therefore we aim to show that
\[ \epsilon_{\cT} = \frac{1}{2} \max_{y \in \cX} \left( \sum_{t \in \cT} \gamma_t \right)^{-1} \sum_{t \in \cT} \gamma_t g(x_t)^\top (y-x_t) \]
is small. Note also that the weights $\gamma_t$ used in \cref{eq:mirror-prox-basic-bound} will be chosen to be exactly the step sizes in \cref{mirror-prox-non-differentiable}. To begin our analysis, we develop a bound for single iteration of \cref{mirror-prox-non-differentiable}.
\begin{lemma}\label{lemma:mp-single-iter}
Given $v_t \in \cX$, let $x_t=\Prox_{v_t}(-\gamma_t \tilde{g}_t)$ and $v_{t+1}=\Prox_{v_t}(-\gamma_t g_t)$ be the points computed at iteration $t$ of \cref{mirror-prox-non-differentiable}. Then
\[ \gamma_t g_t^\top (x-x_t) \leq V_{v_t}(x) - V_{v_{t+1}}(x) +  \frac{1}{2} \left(\gamma_t^2 \| g_t - \tilde{g}_t \|_*^2 - \|x_t - v_t\|^2 \right). \]
\end{lemma}
\begin{proof}{Proof.}
By the optimality condition for $x_t = \Prox_{v_t}(-\gamma_t \tilde{g}_t) = \argmin_{z\in\cX} \left\{ -(\gamma_t \tilde{g}_t + \grad \Phi(v_t))^\top z + \Phi(z) \right\}$, we have for every $x\in\cX$
\begin{align*}
0 &\leq \left( -\gamma_t \tilde{g}_t-\grad\Phi(v_t)+\grad\Phi(x_t) \right)^\top (x - x_t)\\
-\gamma_t \tilde{g}_t^\top (x_t - x) &\leq \left( \grad \phi(x_t) - \grad \Phi(v_t) \right)^\top (x - x_t)\\
&= V_{v_t}(x) - V_{x_t}(x) - V_{v_t}(x_t).
\end{align*}
Similarly, by the optimality condition for $v_{t+1} = \Prox_{v_t}(-\gamma_t g_t) = \argmin_{z\in\cX} \left\{ -(\gamma_t g_t + \grad \Phi(v_t))^\top z + \Phi(z) \right\}$, we have for every $x\in\cX$
\begin{align*}
-\gamma_t g_t^\top (v_{t+1} - x) &\leq V_{v_t}(x) - V_{v_{t+1}}(x) - V_{v_t}(v_{t+1}).
\end{align*}
Set $x = v_{t+1}$ in the first inequality to get $-\gamma_t g_t^\top (x_t - v_{t+1}) \leq V_{v_t}(v_{t+1}) - V_{x_t}(v_{t+1}) - V_{v_t}(x_t)$, then add to the second inequality:
\begin{align*}
-\gamma_t \tilde{g}_t^\top (x_t - v_{t+1}) - \gamma_t g_t^\top (v_{t+1} - x) &\leq V_{v_t}(x) - V_{v_{t+1}}(x) - V_{x_t}(v_{t+1}) - V_{v_t}(x_t)\\
-\gamma_t g_t^\top (x_t - x) &\leq V_{v_t}(x) - V_{v_{t+1}}(x) - V_{x_t}(v_{t+1}) - V_{v_t}(x_t) + \gamma_t \left( g_t - \tilde{g}_t \right)^\top (v_{t+1} - x_t)\\
&\leq V_{v_t}(x) - V_{v_{t+1}}(x) - V_{v_t}(x_t) + \gamma_t \left( g_t - \tilde{g}_t \right)^\top (v_{t+1} - x_t) - \frac{1}{2} \|v_{t+1} - x_t\|^2\\
&\leq V_{v_t}(x) - V_{v_{t+1}}(x) - V_{v_t}(x_t) + \gamma_t \| g_t - \tilde{g}_t \|_* \|v_{t+1} - x_t\| - \frac{1}{2} \|v_{t+1} - x_t\|^2\\
&\leq V_{v_t}(x) - V_{v_{t+1}}(x) - V_{v_t}(x_t) + \frac{\gamma_t^2}{2} \| g_t - \tilde{g}_t \|_*^2
\end{align*}
where the third inequality follows since $\Phi$ is 1-strongly convex, and the last inequality follows since $\gamma_t \| g_t - \tilde{g}_t \|_* \|v_{t+1} - x_t\| - \frac{1}{2} \|v_{t+1} - x_t\|^2 \leq \frac{\gamma_t^2}{2} \| g_t - \tilde{g}_t \|_*^2$. Then again by 1-strong convexity of $\Phi$, we have
\[ -\gamma_t g_t^\top (x_t - x) \leq V_{v_t}(x) - V_{v_{t+1}}(x) +  \frac{1}{2} \left(\gamma_t^2 \| g_t - \tilde{g}_t \|_*^2 - \|x_t - v_t\|^2 \right) \]
as required.
\Halmos
\end{proof}

The next lemma incorporates H\"{o}lder-smoothness of $F$ into our bound.
\begin{lemma}\label{lemma:mp-holder-smooth}
Suppose $F$ is H\"{o}lder-smooth (\cref{def:holder-smooth}) with function $h(z) = \sum_{j \in [k]} \beta_j z^{\sigma_j}$; define $\sigma = \min_{j\in[k]} \sigma_j$, $\bar{\sigma} = \max_{j \in [k]} \sigma_j$, and assume that $0 \leq \sigma \leq \bar{\sigma} \leq 1$. Let $g(x_t) \in \partial^{\uparrow}(x_t)$, $g(v_t) \in \partial^{\uparrow}(v_t)$ be selections of up-super-gradients such that
\[ \|g(x_t) - g(v_t)\|_* \leq h(\|x_t - v_t\|). \]
Furthermore, suppose that $\|g_t - g(x_t)\|_* \leq \delta$ and $\|\tilde{g}_t - g(v_t)\|_* \leq \delta$, and that $\gamma_t = 1/\left( \sqrt{2} t^{\frac{1-\sigma}{2}} \sum_{j \in [k]} \beta_j \right)$. Then
\[ \gamma_t^2 \| g_t - \tilde{g}_t \|_*^2 - \|x_t - v_t\|^2 = \frac{4\delta^2}{t^{1-\sigma} \left( \sum_{j \in [k]} \beta_j \right)^2} + \frac{1}{t}. \]
\end{lemma}
\begin{proof}{Proof.}
First observe that
\begin{align*}
\| g_t - \tilde{g}_t \|_* &\leq \|g(x_t) - g(v_t)\|_* + \|g_t - g(x_t)\|_* + \|\tilde{g}_t - g(v_t)\|_*\\
&\leq h(\|x_t - v_t\|_*) + 2\delta\\
&= \sum_{j \in [k]} \beta_j \|x_t - v_t\|_*^{\sigma_j} + 2\delta.
\end{align*}
Therefore
\begin{align*}
\gamma_t^2 \| g_t - \tilde{g}_t \|_*^2 &\leq \gamma_t^2 \left( \sum_{j \in [k]} \beta_j \|x_t - v_t\|_*^{\sigma_j} + 2\delta \right)^2 \leq 8 \delta^2 \gamma_t^2 + 2 \gamma_t^2 \left( \sum_{j \in [k]} \beta_j \|x_t - v_t\|_*^{\sigma_j} \right)^2.
\end{align*}
Now, we have $\gamma_t^2 \| g_t - \tilde{g}_t \|_*^2 - \|x_t - v_t\|^2 = 8 \delta^2 \gamma_t^2 + 2 \gamma_t^2 \left( \sum_{j \in [k]} \beta_j \|x_t - v_t\|^{\sigma_j} \right)^2 - \|x_t - v_t\|^2$. Examining the terms which involve $\|x_t - v_t\|$, and recalling that $\gamma_t = 1/\left( \sqrt{2} t^{\frac{1-\sigma}{2}} \sum_{j \in [k]} \beta_j \right)$, we have
\begin{align*}
2 \gamma_t^2 \left( \sum_{j \in [k]} \beta_j \|x_t - v_t\|^{\sigma_j} \right)^2 - \|x_t - v_t\|^2 &= \frac{1}{t^{1-\sigma}} \left( \frac{\sum_{j \in [k]} \beta_j \|x_t - v_t\|^{\sigma_j}}{\sum_{j \in [k]} \beta_j} \right)^2 - \|x_t - v_t\|^2.
\end{align*}
Define
\[ s_t(r) = \frac{1}{t^{1-\sigma}} \left( \frac{\sum_{j \in [k]} \beta_j r^{\sigma_j}}{\sum_{j \in [k]} \beta_j} \right)^2 - r^2. \]
We are interested in bounding $s_t(\|x_t-v_t\|)$, which can be bounded by $\sup_{r \geq 0} s_t(r)$. For $r \leq 1$, we have $s_t(r) \leq \frac{1}{t^{1-\sigma}} r^{2\sigma} - r^2$, and for $r \geq 1$, we have $s_t(r) \leq \frac{1}{t^{1-\sigma}} r^{2\bar{\sigma}} - r^2$. We therefore have
\[ \sup_{r \geq 0} s_t(r) \leq \max\left\{ \sup_{r \geq 0} \left( \frac{1}{t^{1-\sigma}} r^{2\sigma} - r^2 \right), \sup_{r \geq 0} \left( \frac{1}{t^{1-\sigma}} r^{2\bar{\sigma}} - r^2 \right) \right\}. \]
Let $\tilde{s}_t(r) = \frac{1}{t^{1-\sigma}} r^{2\tilde{\sigma}} - r^2$ for any $\tilde{\sigma} \in [\sigma, \bar{\sigma}]$. Note that if $\tilde{\sigma}=1$ then $\tilde{s}_t(r) = r^2 (1/t^{1-\sigma} - 1) \leq 0$ since $t \geq 1$. Assume that $\tilde{\sigma} < 1$. Then
\[ \tilde{s}_t'(r) = 2 r \left( \frac{\tilde{\sigma}}{t^{1-\sigma}} r^{2(\tilde{\sigma}-1)} - 1 \right). \]
When $r \leq (t^{1-\sigma}/\tilde{\sigma})^{\frac{1}{2(\tilde{\sigma}-1)}}$ then $\tilde{s}_t'(r) \geq 0$ and $r \geq (t^{1-\sigma}/\tilde{\sigma})^{\frac{1}{2(\tilde{\sigma}-1)}}$ then $\tilde{s}_t'(r) \leq 0$, therefore $\tilde{s}_t(r)$ is maximized at $r = (t^{1-\sigma}/\tilde{\sigma})^{\frac{1}{2(\tilde{\sigma}-1)}}$, with value
\[ \max_{r \geq 0} \tilde{s}_t(r) = \frac{1}{t^{1-\sigma}} \frac{t^{\frac{\tilde{\sigma}(1-\sigma)}{\tilde{\sigma}-1}}}{\tilde{\sigma}^{\frac{\tilde{\sigma}}{\tilde{\sigma}-1}}} - \frac{t^{\frac{1-\sigma}{\tilde{\sigma}-1}}}{\tilde{\sigma}^{\frac{1}{\tilde{\sigma}-1}}} = \frac{1}{t^{\frac{1-\sigma}{1-\tilde{\sigma}}}} \left( \tilde{\sigma}^{\frac{\tilde{\sigma}}{1-\tilde{\sigma}}} - \tilde{\sigma}^{\frac{1}{1-\tilde{\sigma}}} \right). \]
Since $0 \leq \sigma \leq \tilde{\sigma} \leq 1$, $\frac{1-\sigma}{1-\tilde{\sigma}} \geq 1$ and $\tilde{\sigma}^{\frac{\tilde{\sigma}}{1-\tilde{\sigma}}} - \tilde{\sigma}^{\frac{1}{1-\tilde{\sigma}}} \leq 1$, therefore
\[ \max_{r \geq 0} \tilde{s}_t(r) = \frac{1}{t^{\frac{1-\sigma}{1-\tilde{\sigma}}}} \left( \tilde{\sigma}^{\frac{\tilde{\sigma}}{1-\tilde{\sigma}}} - \tilde{\sigma}^{\frac{1}{1-\tilde{\sigma}}} \right) \leq \frac{1}{t}. \]
But this implies that
\[ 2 \gamma_t^2 \left( \sum_{j \in [k]} \beta_j \|x_t - v_t\|^{\sigma_j} \right)^2 - \|x_t - v_t\|^2 \leq \sup_{r \geq 0} s_t(r) \leq \max\left\{ \sup_{r \geq 0} \left( \frac{1}{t^{1-\sigma}} r^{2\sigma} - r^2 \right), \sup_{r \geq 0} \left( \frac{1}{t^{1-\sigma}} r^{2\bar{\sigma}} - r^2 \right) \right\} \leq \frac{1}{2t}. \]
Thus
\[ \gamma_t^2 \| g_t - \tilde{g}_t \|_*^2 - \|x_t - v_t\|^2 = \frac{4\delta^2}{t^{1-\sigma} \left( \sum_{j \in [k]} \beta_j \right)^2} + \frac{1}{t} \]
as required.
\Halmos
\end{proof}

\begin{corollary}\label{cor:mp-bound}
Under the assumptions of \cref{lemma:mp-single-iter,lemma:mp-holder-smooth}, and additionally assuming that $\max_{x,y \in \cX} \|x-y\| \leq D$, we have the following bound on iterates of \cref{mirror-prox-non-differentiable}: for any $y \in \cX$,
\begin{align*}
\gamma_t g(x_t)^\top (y - x_t) \leq V_{v_t}(y) - V_{v_{t+1}}(y) + \gamma_t \delta D + \frac{2\delta^2}{t^{1-\sigma} \left( \sum_{j \in [k]} \beta_j \right)^2} + \frac{1}{2t}.
\end{align*}
\end{corollary}
\begin{proof}{Proof.}
First we combine the bounds in \cref{lemma:mp-single-iter,lemma:mp-holder-smooth} to get for any $y \in \cX$
\begin{align*}
	\gamma_t g_t^\top (y - x_t) \leq V_{v_t}(y) - V_{v_{t+1}}(y) + \frac{2\delta^2}{t^{1-\sigma} \left( \sum_{j \in [k]} \beta_j \right)^2} + \frac{1}{2t}.
\end{align*}
Now observe that
\begin{align*}
\gamma_t g(x_t)^\top (y - x_t) &= \gamma_t (g_t - g_t + g(x_t))^\top (y - x_t)\\
&= \gamma_t g_t^\top (y - x_t) + \gamma_t (g(x_t) - g_t)^\top (y - x_t)\\
&\leq \gamma_t g_t^\top (y - x_t) + \gamma_t \|g(x_t) - g_t\|_* \|y-x_t\|\\
&\leq \gamma_t g_t^\top (y - x_t) + \gamma_t \delta D,
\end{align*}
where the first inequality follows from H\"{o}lder's inequality and the second follows since $\|g(x_t) - g_t\|_* \leq \delta$, $\|y-x_t\| \leq D$. Combining this with the above inequality gives the result.
\Halmos
\end{proof}

We are now ready to give our convergence result.
\begin{theorem}\label{thm:mprox-complexity}
	Suppose the following hold:
	\begin{itemize}
		\item $F$ is monotone and up-concave, and $\max_{x,y \in \cX} \|x-y\| \leq D$.
		
		\item $F$ is H\"{o}lder-smooth (\cref{def:holder-smooth}) with function $h(z) = \sum_{j \in [k]} \beta_j z^{\sigma_j}$.
		
		\item $0 \leq \sigma \leq \max_{j \in [k]} \sigma_j \leq 1$, where $\sigma := \min_{j\in[k]} \sigma_j$.
		
		\item Furthermore, assume that $\sup_{x \in \cX} V_{v_1}(x) \leq D$.
	\end{itemize}
	In \cref{mirror-prox-non-differentiable}, at each iteration $t$ choose $\gamma_t = 1/\left( \sqrt{2} t^{\frac{1-\sigma}{2}} \sum_{j \in [k]} \beta_j \right)$, and $\tilde{g}_t,g_t$ such that $\|\tilde{g}_t - g(v_t)\|_* \leq \delta$ and $\|g_t - g(x_t)\|_* \leq \delta$, where $g(x_t) \in \partial^{\uparrow}(x_t)$, $g(v_t) \in \partial^{\uparrow}(v_t)$ are selections of up-super-gradients such that $\|g(x_t) - g(v_t)\|_* \leq h(\|x_t - v_t\|)$.  Then there exists some constant $c_\sigma > 0$ depending only on $\sigma$ such that
	\begin{align*}
		\frac{1}{2} \max_{x \in \cX} F(x) &\leq \max_{t \in [T]} F(x_t) + 2 \sqrt{2} \left( \sum_{j \in [k]} \beta_j \right) \frac{D}{T^{\frac{1+\sigma}{2}}} + \frac{\delta D}{2} + \frac{4 \sqrt{2} \delta^2}{\sigma \left(\sum_{j \in [k]} \beta_j\right) T^{\frac{1-\sigma}{2}}}\\
		&\quad + 4 \sqrt{2} \left(\sum_{j \in [k]} \beta_j \right) \frac{1}{T^{\frac{1+\sigma}{2}}}.
	\end{align*}
\end{theorem}
\begin{proof}{Proof.}
\cref{lemma:hassani_subgrad} gives \eqref{eq:mirror-prox-basic-bound} which states that
\begin{align*}
	\frac{1}{2} \max_{x \in \cX} F(x) \leq \max_{t \in [T]} F(x_t) + \frac{1}{2} \max_{y \in \cX} \left( \sum_{t \in [T]} \gamma_t \right)^{-1} \sum_{t \in [T]} \gamma_t g(x_t)^\top (y-x_t).
\end{align*}
Summing the bound in \cref{cor:mp-bound} over $t \in [T]$, we get
\begin{equation}\label{eq:mp-proof-bound}
\sum_{t \in \cT} \gamma_t g(x_t)^\top (y-x_t) \leq V_{v_1}(y) + \left( \sum_{t \in \cT} \gamma_t \right) \delta D + \left( \frac{\sqrt{2} \delta}{\sum_{j \in [k]} \beta_j} \right)^2 \sum_{t \in \cT} \frac{1}{t^{1-\sigma}} + \frac{1}{2} \sum_{t \in \cT} \frac{1}{t}.
\end{equation}

To bound the sum $\sum_{t \in \cT}\gamma_t$, we consider $\sum_{t \in \cT} t^{-\frac{1-\sigma}{2}}$. Let $T_0 = \lfloor (T-2)/3 \rfloor +1$, then $\cT = \{T_0,\ldots,T-1\}$.
Since $1-\sigma\geq 0$, we know that $1/t^{\frac{1-\sigma}{2}}$ is a decreasing function. Noting that $T_0 \leq (T+1)/3 \leq T/2$ for $T \geq 1$,
\begin{equation*}
	\sum_{t=T_0}^T t^{-\frac{1-\sigma}{2}} \geq \int_{T_0}^T t^{-\frac{1-\sigma}{2}} dt  = \frac{2}{1+\sigma} \left( T^{\frac{1+\sigma}{2}} - T_0^{\frac{1+\sigma}{2}} \right) \geq T^{\frac{1+\sigma}{2}} - (T/2)^{\frac{1+\sigma}{2}} \geq \frac{T^{\frac{1+\sigma}{2}}}{4}.
\end{equation*}
Similarly, we have
\begin{align*}
	\sum_{t=T_0}^{T-1} \frac{1}{t^{1-\sigma}} &\leq \int_{T_0-1}^{T-1} t^{-(1-\sigma)} dt = \frac{(T-1)^\sigma - \lfloor (T-2)/3 \rfloor^\sigma}{\sigma} \leq \frac{T^\sigma}{\sigma}\\
	\sum_{t=T_0}^{T-1} \frac{1}{t} &\leq \frac{T-1-T_0}{T_0} = \frac{T-1-\lfloor (T-2)/3 \rfloor}{\lfloor (T-2)/3 \rfloor+1} \leq 2.
\end{align*}
Dividing both sides of \eqref{eq:mp-proof-bound} by $\sum_{t \in [T]}\gamma_t$ and substituting in the lower bound, as well as the upper bounds for the other sum terms, we get
\begin{align*}
\left( \sum_{t \in [T]} \gamma_t \right)^{-1} \sum_{t \in [T]} \gamma_t g(x_t)^\top (y-x_t) &\leq 4 \sqrt{2} \left( \sum_{j \in [k]} \beta_j \right) \frac{D}{T^{\frac{1+\sigma}{2}}} + \delta D + \frac{8 \sqrt{2} \delta^2}{\sigma \left(\sum_{j \in [k]} \beta_j\right) T^{\frac{1-\sigma}{2}}}\\
&\quad + 8 \sqrt{2} \left(\sum_{j \in [k]} \beta_j \right) \frac{1}{T^{\frac{1+\sigma}{2}}}. 
\end{align*}
This gives the result.
\Halmos
\end{proof}

By Theorem~\ref{thm:mprox-complexity}, we obtain the following approximation guarantee for using Algorithm~\ref{mirror-prox-non-differentiable} to solve \cref{SFM}.
\begin{corollary}\label{cor:mp-convergence-rate}
Suppose the conditions of \cref{thm:mprox-complexity} hold. If $\delta = O(\epsilon)$, then Algorithm~\ref{mirror-prox-non-differentiable} returns an $(1/2,\epsilon)$-approximate solution after $O(1/\epsilon^{2/(1+\sigma)})$ iterations.
\end{corollary}
By Corollary~\ref{cor:mp-convergence-rate}, Algorithm~\ref{mirror-prox-non-differentiable} converges to an $(1/2,\epsilon)$-approximate solution after a finite number of iterations even when $\min_{j\in[k]}\sigma_j=0$.

\section{Robust DR-submodular maximization.}\label{sec:robust}

In this section, we revisit \eqref{eq:robust-SFM-c} from \cref{sec:prelim-motivation}. Recall that the objective from \cref{eq:robust-SFM-c} is
\begin{align}\label{eq:robust-SFM-c-objective}
F(x) := \min_{p \in \cP} f(x, p),
\end{align}
where $f(x,p)$ is a class of functions parametrized by a given set $\cP$. We examine the application of non-smooth methods from \cref{sec:mp} as well as smooth methods from \cref{sec:cg} after appropriate smoothing.

\subsection{Non-smooth approach via \cref{mirror-prox-non-differentiable}.} We state the precise assumptions on our objective function as follows.
\begin{assumption}\label{ass:robust-SFM}
For each parameter $p \in \cP$, $f(x,p)$ is monotone, non-negative and DR-submodular. Furthermore, each $f(x,p)$ is differentiable and $L$-Lipschitz continuous with respect to the norm $\|\cdot\|$ for some $L>0$, i.e.,
$|f(x,p)- f(y,p)|\leq L\|x-y\|$ for any $x,y\in\mathbb{R}^d$ and each $i\in[n]$. In particular, $L$ is independent of the parameter $p \in \cP$.
\end{assumption}

Under \cref{ass:robust-SFM}, the objective \eqref{eq:robust-SFM-c-objective} satisfies a number of useful properties.
\begin{lemma}\label{lemma:robust-SFM-objective}
The following hold for $F$ in \cref{eq:robust-SFM-c-objective} under \cref{ass:robust-SFM}:
\begin{itemize}
\item $F$ is monotone, non-negative and up-concave.

\item For any given $x$, up-super-gradients may be found given a minimization oracle for $f(x,p)$ over $p \in \cP$:
\[ \grad_{x} f(x, p_{x}^*) \in \partial^\uparrow F(x) \]
for any $p_{x}^* \in \cP$ that minimizes $f(x,p)$.

\item $\|\grad_{x} f(x, p)\|_* \leq L$ for any $x \in \cX$ and $p \in \cP$, and $F$ is $L$-Lipschitz continuous, i.e., it satisfies~\eqref{eq:holder-smooth-nd-1} with norm $\|\cdot\|$ and function $h=2L$, a constant function.
\end{itemize}
\end{lemma}
We provide the proof in Appendix~\ref{proofs:robust}.

Due to \cref{lemma:robust-SFM-objective}, we may apply Algorithm~\ref{mirror-prox-non-differentiable} for solving \eqref{eq:robust-SFM-c}.
Under a particular speficiation of $g_t$ and $g_{t+1/2}$ in \cref{mirror-prox-non-differentiable}, Theorem~\ref{thm:mprox-complexity} and Corollary~\ref{cor:mp-convergence-rate} gives following convergence result. 
\begin{theorem}\label{thm:robust-submodular-guarantee}
Suppose \cref{ass:robust-SFM} holds, and we run \cref{mirror-prox-non-differentiable} with $\tilde{g}_t = \grad_{x} f(v_t, \tilde{p}_t)$ and $g_t = \grad_{x} f(x_t, p_t)$, where $\tilde{p}_t,p_t$ are minimizers of $f(v_t,p)$ and $f(x_t,p)$ over $p \in \cP$ respectively. Then \cref{mirror-prox-non-differentiable} returns a solution $\hat{x}\in\mathcal{X}$ such that
$$F(\hat{x}) \geq \frac{1}{2}\sup_{x\in\mathcal{X}} F(x) - \frac{24(D+1)L}{\sqrt{T}}$$
where $\sup_{x,y} V_{y}(x)\leq D$. Therefore $\hat{x}$ is a $(1/2,\epsilon)$-approximate solution to \eqref{eq:robust-SFM-c}, obtained after $O(1/\epsilon^{2})$ iterations.
\end{theorem}
When $f(x,p) = \sum_{i \in [n]} p_i f_i(x)$ and $\cP$ is a set over which exact linear minimization is easy, then we do not need to account for errors in the up-super-gradient computation. This is the case, e.g., for the sets examined by \citet{equator,dro-sfm}. However, if we cannot perform exact minimization, but instead can perform approximate minimization, then the extra error terms in \cref{thm:mprox-complexity,cor:mp-convergence-rate} come into effect.

\subsubsection{Usage in discrete multiobjective problem.} Suppose we are given monotone submodular functions $f_{1,0},\ldots,f_{n,0}$ over ground set $V = [d]$ as well as values $v_1,\ldots,v_n \in \bbR$. Given a matroid $\cM$, we would like to find a set $S \in \cM$ such that $f_{i,0}(S) \geq v_i$ for all $i \in [n]$, or certify that none exists. This multiobjective problem has been studied by \citet{swap-rounding,Udwani2021}. We will show how usage of the mirror-prox method (\cref{mirror-prox-non-differentiable}) can be used within the framework of \citep{swap-rounding,Udwani2021} to obtain fast algorithms with improved guarantees, in the setting where we have cardinality constraints $\cM = \{S \subseteq [d] : |S| \leq k\}$.

We now describe the framework from \citep{swap-rounding,Udwani2021}, and point out adaptations where the mirror-prox method may be used (specifically, for implementing step 2 below). We first transform the problem slightly by setting $v_i \leftarrow \max\{0,v_i - f_{i,0}(\emptyset)\}$ and $f_i(S) = \min\{v_i, f_{i,0}(S) - f_{i,0}(\emptyset)\}$, which guarantees $0 \leq f_i(S) \leq v_i$ for each $i \in [n]$ and $S \subseteq V$. If there exists a set $S \in \cM$ solving the problem, then $f_i(S) = v_i$ for all $i \in [n]$. If $v_i = 0$, then we can ignore this particular $v_i$, hence define $v = \min_{i \in [n]} v_i$, and assume $v > 0$.

\paragraph{Step 1: guessing a good starting set.} For some $\epsilon > 0$ to be chosen later, we seek a set $S_1$ such that $f_i(S_1 \cup \{j\}) - f_i(S_1) \leq \epsilon^3 v_i$ for all $i \in [n]$ and $j \in [d]$. To do this, we start with $S_1 = \emptyset$, and pass through the set $V$ in order, adding $j$ to $S_1$ if there exists some $i \in [n]$ such that $f_i(S_1 \cup \{j\}) - f_i(S_1) > \epsilon^3 v_i$. Let $j_1 < \ldots < j_{s_i}$ be the elements in $S_1$ that were added due to function $i$, i.e., $f_i(\{j_1,\ldots,j_m,j_{m+1}\}) - f_i(\{j_1,\ldots,j_m\}) > \epsilon^3 v_i$ for $m=0,1,\ldots,s_i-1$. Adding each of these inequalities up, we have
\[ f_i(S_1) \geq f_i(\{j_1,\ldots,j_{s_i}\}) - f_i(\emptyset) > \epsilon^3 v_i s_i. \]
Since $f_i(S_1) \leq v_i$ by construction, we must have $s_i \leq 1/\epsilon^3$, and hence $|S_1| \leq n/\epsilon^3$. Note that we will choose $\epsilon$ in such a way that $n/\epsilon^3 \leq k$.

In later steps, we will augment $S_1$ with a random set $S$ with high probability guarantees that $f_i(S \cup S_1) \geq \gamma v_i$ for some $\gamma \in (0,1)$ to be determined. Since $f_{i,0}$ are assumed monotone, we can ignore functions for which $f_i(S_1) > \beta v_i$, where $\beta \in (0,1)$ will be chosen later. Therefore we assume $f_i(S_1) \leq \beta v_i$ for all $i \in [n]$.

\paragraph{Step 2: use a continuous algorithm to find a good fractional point.} Let $F_i$ be the multilinear extension for $f_i$; note that these are monotone, nonnegative, and DR-submodular functions over $[0,1]^d$. We define marginal functions
\[ G_i(x) := F_i(x \vee \bm{1}(S_1)) - f_i(S_1) \]
for each $i \in [n]$, where $\bm{1}(S_1)$ is the vector for which all indices corresponding to $j \in S_1$ have value $1$, while all other entries have value $0$, and $x \vee \bm{1}(S_1)$ is the component-wise maximum of $x$ and $\bm{1}(S_1)$. Each $G_i$ is also monotone, nonnegative, and DR-submodular. We define $\cX = \{x \in [0,1]^d : \bm{1}^\top x \leq k \}$ to be the polymatroid associated with $\cM$. We now aim to optimize the robust DR-submodular function
\begin{align*}
	\max_{x : x \vee \bm{1}(S_1) \in \cX} \min_{i \in [n]} \frac{G_i(x)}{v_i}.
\end{align*}
\citet{Udwani2021} provides an algorithm based on multiplicative weights which achieves an $((1-1/e)^2, \delta)$-approximation in $O(1/\delta^2)$ iterations. On the other hand, using the mirror-prox method (\cref{thm:robust-submodular-guarantee}), we may achieve a $(1/2, \delta)$-approximation in $O(1/\delta^2)$ iterations (noting that $(1-1/e)^2 \approx 0.4 < 1/2$). Assume now that we use any algorithm that achieves a $(\alpha,\delta)$-guarantee. Then we find a point $\bar{x}$ such that
\[ \frac{G_i(\bar{x})}{v_i} \geq \alpha \max_{x : x \vee \bm{1}(S_1) \in \cX} \min_{i \in [n]} \frac{G_i(x)}{v_i} - \delta. \]
The next lemma relates the lower bound to our problem.
\begin{lemma}\label{lemma:fractional-bound}
	Suppose $F$ is DR-submodular with $F(0) = 0$. Then for any $x \in [0,1]^d$ with $\|x\|_1 = k$ and any $k_1 \in [0,k]$ we have
	\[ F((k_1/k) x) \geq (k_1/k) F(x). \]
	Assume that $|S_1| \leq k$. Then at least one of the following statements are true (they may both be true):
	\begin{itemize}
		\item there is no set $S \in \cM$ for which $f_i(S) \geq v_i$ for all $i \in [n]$; or
		\item there exists $\tilde{x}$ such that $\tilde{x} \vee \bm{1}(S_1) \in \cX$ and for each $i \in [n]$ we have
		\[ G_i(\tilde{x}) \geq \left( 1-|S_1|/k \right) (v_i - f_i(S_1)), \]
		and consequently we have $\max_{x : x \vee \bm{1}(S_1) \in \cX} \min_{i \in [n]} \frac{G_i(x)}{v_i} \geq \left( 1-|S_1|/k \right) \min_{i \in [n]} \left( 1 - f_i(S_1)/v_i \right)$.
	\end{itemize}
\end{lemma}
\begin{proof}{Proof of \cref{lemma:fractional-bound}.}
	The first inequality follows from up-concavity of $F$ and because $x \geq 0$.
	
	Consider any $S \in \cM$ and vector $\bm{1}(S \setminus S_1)$. We have $k_2 := \|\bm{1}(S \setminus S_1)\|_1 \leq k$, while $\|\bm{1}(S \setminus S_1) \vee \bm{1}(S_1)\|_1 = k_2 + |S_1| \geq k$. Notably,
	\[ \left\| \left( \frac{k-|S_1|}{k_2} \bm{1}(S \setminus S_1) \right) \vee \bm{1}(S_1) \right\|_1 = k \quad \text{and} \quad \frac{k-|S_1|}{k_2} \geq \frac{k-|S_1|}{k}. \]
	Apply the first inequality to $G_i(x) = F_i(x \vee \bm{1}(S_1)) - f_i(S_1)$, we get
	\begin{align*}
		G_i\left( \frac{k-|S_1|}{k_2} \bm{1}(S \setminus S_1) \right) &\geq \frac{k-|S_1|}{k_2} G_i\left( \bm{1}(S \setminus S_1) \vee \bm{1}(S_1) \right)\\
		&= \frac{k-|S_1|}{k_2} (f_i\left( S \cup S_1 \right) - f_i(S_1))\\
		&\geq \frac{k-|S_1|}{k} (f_i\left( S \cup S_1 \right) - f_i(S_1)).
	\end{align*}
	Choosing $S \in \cM$ such that $f_i(S) \geq v_i$ and dividing both sides by $v_i$ for all $i \in [n]$, we get the result.
	\Halmos
\end{proof}
By \cref{lemma:fractional-bound} the point $\bar{x}$ satisfies
\begin{align*}
	\frac{G_i(\bar{x})}{v_i} &\geq \alpha (1-|S_1|/k) \min_{i \in [n]} (1 - f_i(S_1)/v_i) - \delta\\
	&\geq \alpha (1-|S_1|/k) (1-\beta) - \delta
\end{align*}
where the second inquality follows as $f_i(S_1)/v_i \leq \beta$ for all $i \in [n]$.

\paragraph{Step 3: rounding.} Let $g_i(S) = f_i(S \cup S_1) - f_i(S_1)$. Note that $G_i$ are multilinear extensions of $g_i$, and that $g_i(\{j\}) \leq \epsilon^3 v_i$ for all $i \in [n]$ and $j \in [d]$ by construction of $S_1$. The swap rounding algorithm of \citet{swap-rounding} finds a random set $S$ such that $S \cup S_1 \in \cM$ (with probability 1) and
\begin{align*}
	\bbP\left[ g_i(S) < (1-\epsilon) G_i(\bar{x}) \right] &= \bbP\left[ \frac{g_i(S)}{\epsilon^3 v_i} > (1-\epsilon) \frac{G_i(\bar{x})}{\epsilon^3 v_i} \right]\\
	&\leq \exp\left( -\frac{G_i(\bar{x})}{8 \epsilon v_i} \right)\\
	&\leq \exp\left( -\frac{\alpha (1-|S_1|/k) (1-\beta) - \delta}{8 \epsilon} \right)
\end{align*}
where the last inequality follows since $\frac{G_i(\bar{x})}{v_i} \geq \alpha (1-|S_1|/k) (1-\beta) - \delta$.

\paragraph{Guarantee.} Applying a union bound and taking the complement, we get
\begin{align*}
	&\bbP\left[ g_i(S) \geq (\alpha (1-\epsilon)(1-n/(k \epsilon^3)) (1-\beta) - \delta) v_i \ \forall i \in [n] \right]\\
	&\geq \bbP\left[ g_i(S) \geq (\alpha (1-\epsilon) (1-|S_1|/k) (1-\beta) - \delta) v_i \ \forall i \in [n] \right]\\
	&\geq \bbP\left[ g_i(S) \geq (1-\epsilon) G_i(\bar{x}) \ \forall i \in [n] \right]\\
	&\geq 1 - n \exp\left( -\frac{\alpha (1-|S_1|/k) (1-\beta) - \delta}{8 \epsilon} \right)\\
	&\geq 1 - n \exp\left( -\frac{\alpha (1-n/(k\epsilon^3)) (1-\beta) - \delta}{8 \epsilon} \right),
\end{align*}
where the first and last inequalities follow since $|S_1| \leq n/\epsilon^3$. Notice that
\begin{align*}
	\exp\left( -\frac{\alpha (1-n/(k\epsilon^3)) (1-\beta) - \delta}{8 \epsilon} \right) &\leq (1-\rho)/n\\
	\frac{\alpha (1-n/(k\epsilon^3)) (1-\beta) - \delta}{8 \epsilon} &\geq \log(n) - \log(1-\rho)\\
	\frac{\alpha (1-\beta) - \delta}{8 \epsilon} - \frac{n}{k} \frac{\alpha (1-\beta)}{8 \epsilon^4} &\geq \log(n) - \log(1-\rho)\\
	\epsilon^3 \left( \frac{\alpha (1-\beta) - \delta}{8 (\log(n) - \log(1-\rho))} - \epsilon \right) &\geq \frac{n}{k} \frac{\alpha (1-\beta)}{8 (\log(n) - \log(1-\rho))}\\
	\iff \epsilon^3(C_n-D_n - \epsilon) &\geq C_n (n/k)\\
	&\quad \text{ where } C_n := \frac{\alpha (1-\beta)}{8 (\log(n) - \log(1-\rho))}, \ D_n = \frac{\delta}{8 (\log(n) - \log(1-\rho))}.
\end{align*}
If $\epsilon = 3 (C_n - D_n) /4 = O(1/\log(n))$, then the left hand side maximized, so that
\[ \frac{27 (C_n - D_n)^4}{256} \geq C_n (n/k) \iff \frac{n}{C_n^3 (1 - D_n/C_n)^4 k} \leq \frac{27}{256}. \]
If $\frac{n \log^3(n)}{k} = o(1)$ and $\delta < \alpha(1-\beta)$, then this is guaranteed in the limit. Furthermore, $1-(n/(k \epsilon^3)) \to 1$ as well. In summary, we have the following guarantee.

\begin{theorem}
	Let $\epsilon = 3 (C_n - D_n) /4 = \frac{3 (\alpha (1-\beta) - \delta)}{32 (\log(n) - \log(1-\rho))}$. Then the procedure generates a random $S$ such that $S \cup S_1 \in \cM$ with probability $1$ and with probability $\rho > 0$ we have
	\[ \forall i \in [n], \quad f_i(S \cup S_1) \geq (\alpha h(n,k) (1-\beta) - \delta) v_i + (1-\alpha h(n,k) (1-\beta) + \delta) f_i(S_1) \quad \text{or} \quad f_i(S \cup S_1) \geq \beta v_i, \]
	where $h(n,k)$ is a function such that $h(n,k) \to 1$ when $\frac{n \log^3(n)}{k} = o(1)$.
\end{theorem}

\subsection{Smooth approach via \cref{alg:conditional-gradient-upconcave-monotone} for distributionally robust DR-submodular maximization under Wasserstein ambiguity.}\label{sec:dro}

We now revisit the distributionally robust model \eqref{eq:DR-SFM} from \cref{sec:prelim-motivation}, where we have a class of DR-submodular functions $\tilde{f}(x,\xi)$ with parameters $\xi \in \Xi$. Recall that we have a distribution $Q=\sum_{i\in[n]} q_i\delta_{\xi^i}$ over $\Xi$ and an  ambiguity set $\cB$ defined by the 2-Wasserstein ball of radius $\theta$ around $Q$. By \cref{prop:reformulation-1}, this admits a reformulation into the form of \eqref{eq:robust-SFM-c}, where the objective is defined by \cref{eq:DR-SFM-wasserstein-form}:
\begin{subequations}\label{eq:DR-SFM-wasserstein}
\begin{align}
\max_{x \in \cX} F(x), &\quad \text{where } F(x) := \inf_{p \in \cP} f(x,p),\\
 p &:= (\zeta^1;\cdots;\zeta^n) \in \cP := \left\{(\zeta^1;\cdots;\zeta^n): \sum_{i \in [n]} q_i \|\xi^i - \zeta^i\|_2^2 \leq \theta^2,\ \zeta^i\in\Xi\ \ \forall i\in[n] \right\}\\
f(x,p) &:= \sum_{i \in [n]} q_i \tilde{f}(x, \zeta^i).
\end{align}
\end{subequations}
We will impose some assumptions on our setting. Denote
\begin{equation}\label{convex-domain:gamma}
	\Gamma=\bigcup_{i\in[n]}\left\{\zeta\in \Xi:\ q_i\|\xi^i-\zeta\|_2^2\leq \theta^2\right\}.
\end{equation}
Note that $\Gamma$ is bounded as well.
\begin{assumption}\label{ass:DR-SFM}
The functions $\tilde{f}(x,\xi)$ satisfy:
\begin{itemize}
	\item for each $\xi \in \Xi$, $\tilde{f}(x,\xi)$ is monotone, nonnegative, DR-submodular with respect to $x$.
	
	\item for each $x \in \cX$, $\tilde{f}(x,\xi)$ is convex with respect to $\xi$.
	
	\item $\cX$ is bounded, and $\tilde{f}(x,\zeta)$ is $L_2$-Lipschitz continuous in $\zeta$ over $\Gamma$, where $L_2 := 2\sup_{(x,\xi)\in\mathcal{X}\times\Gamma} f(x,\xi) < \infty$. (Note that $L_2$ is bounded since $\cX$ and $\Gamma$ are.)
	
	\item for every $\xi\in\Gamma$, $\tilde{f}(x,\xi)$  is differentiable with respect to $x$ and $L_1$-Lipschitz continuous with respect to $\|\cdot\|$ over $\mathcal{X}$.
	
	\item $\|\grad_{x}\tilde{f}(x,\xi)-\grad_{x}\tilde{f}(y,\xi)\|_*\leq\lambda_1 \|x-y\|$ for any $x,y\in \mathcal{X}$ and $\xi\in\Gamma$.
	
	\item $\|\grad_{x}\tilde{f}(x,\xi)-\grad_{x}\tilde{f}(x,\zeta)\|_*\leq \lambda_2\|\xi-\zeta\|_2$ for any $x\in\mathcal{X}$ and $\xi,\zeta\in \Gamma$.
\end{itemize}
\end{assumption}
We show in Appendix~\ref{dro-discrete} that if $\tilde{f}$ is the multilinear extension of some submodular-convex function, then $F$ satisfies Assumption~\ref{ass:DR-SFM}.

\subsubsection{Smoothing.}\label{sec:preprocessing}

There are a few issues when applying \cref{alg:conditional-gradient-upconcave-monotone,mirror-prox-non-differentiable} to \cref{eq:DR-SFM-wasserstein} under \cref{ass:DR-SFM}. The first issue is that the function $F$ may be non-differentiable, although $\tilde{f}(\cdot,\xi)$ is smooth for every $\xi\in\Gamma$. The second issue is that, even though \cref{lemma:robust-SFM-objective} provides a way to compute up-super-gradients of $F$ by minimizing $f(x,p)$ over $p \in \cP$, but unlike in \cref{sec:robust} this is now a convex nonlinear minimization problem, thus exact solution methods are unlikely. While we can obtain approximate solutions, if $f(x,\cdot)$ is not strongly convex, the distance between the approximate solution and an optimal solution can be large, which potentially incurs a large error in the up-super-gradient computation for $F$.

To remedy the challenges, we construct a function that has favorable structural properties, and at the same time, well approximates the original function $F$. We construct the following function by adding a strongly convex regularization term to $f$:
\begin{align}\label{DR-SFM-H}
	\begin{aligned}
		H_\epsilon(x) &:= \inf_{p \in \cP} h_\epsilon(x,p)\\
		\text{where } h_\epsilon(x,p) &:= f(x,p) + \frac{\epsilon}{2\theta^2} \sum_{i \in [n]} q_i \|\xi^i - \zeta^i\|_2^2.
	\end{aligned}
\end{align}
Here, $\theta$ is the radius of the Wasserstein ball and $\epsilon$ is an accuracy parameter. Note that the function $h_\epsilon$ for any fixed $x$ is strongly convex in $p$ thanks to the regularization term. In addition, $H_\epsilon(x)$ can be interpreted as applying the Nesterov smoothing idea \citep[Section 4.3]{BeckTeboulle2012-smoothing} to $F$.

With $H_\epsilon$ defined as in~\eqref{DR-SFM-H}, we consider approximating~\eqref{eq:DR-SFM-wasserstein} by replacing $F$ with $H_\epsilon$:
\begin{equation}\label{reformulation-re}
	\sup_{x\in\mathcal{X}}H_\epsilon(x).\tag{DR-SFM-c''}
\end{equation}
In fact, solving~\eqref{reformulation-re} is equivalent to solving~\eqref{eq:DR-SFM-wasserstein} up to some additive error. 
\begin{lemma}\label{lemma:approximation}
	Let $\hat{x}$ be an $(\alpha,\epsilon^\prime)$-approximate solution to~\eqref{reformulation-re}. Then 
	$$F(\hat{x})\geq \alpha\cdot\mbox{OPT}-\epsilon^\prime-\frac{\epsilon}{2}
	$$
	where $\mbox{OPT}$ is the value of~\eqref{eq:DR-SFM-wasserstein}, i.e., $\hat{x}$ is an $(\alpha,\epsilon^\prime+\epsilon/2)$-approximate solution to~\eqref{eq:DR-SFM-wasserstein}.
\end{lemma}
\begin{proof}{Proof.}
	We first show that for any $x\in\mathcal{X}$, $F(x)\leq H_\epsilon(x)\leq F(x) + \epsilon/2$. Note that
	\begin{align*}
		H(x)&=\inf_{p\in \cP} \left\{  f(x,p)+ \frac{\epsilon}{2\theta^2} \sum_{i \in [n]} q_i \|\xi^i - \zeta^i\|_2^2 \right\}\\
		&\geq \inf_{p \in \cP} f(x,p) + \frac{\epsilon}{2\theta^2} \inf_{p \in \cP}  \sum_{i \in [n]} q_i \|\xi^i - \zeta^i\|_2^2\\
		&= F(x) + \frac{\epsilon}{2\theta^2} \inf_{p \in \cP} \sum_{i \in [n]} q_i \|\xi^i - \zeta^i\|_2^2.
	\end{align*}
	The last term of this inequality is at least $F(x)$, so $H_\epsilon(x)\geq F(x)$, as required. To show $H_\epsilon(x)\leq F(x) + \epsilon/2$, note that for any $p = (\zeta^1,\ldots,\zeta^n) \in \cP$ we have $h_\epsilon(x,p) = f(x,p) + \frac{\epsilon}{2\theta^2} \sum_{i \in [n]} q_i \|\xi^i - \zeta^i\|_2^2 \leq f(x,p) + \epsilon/2$ by definition of $\cP$. Therefore
	\[ H_\epsilon(x) = \inf_{p \in \cP} h_\epsilon(x,p) \leq \inf_{p \in \cP} f(x,p) + \frac{\epsilon}{2} \leq F(x) + \frac{\epsilon}{2}. \]
	
	Let $\hat{x}$ be an $(\alpha,\epsilon^\prime)$-approximate solution to~\eqref{reformulation-re}. Let $x^*$ be an optimal solution to~\eqref{eq:DR-SFM-wasserstein}, i.e., $F(x^*)=\mbox{OPT}$. Note that
	$$F(\hat{x}) \geq H_\epsilon(\hat{x}) -\frac{\epsilon}{2}\geq  \alpha\cdot \sup_{x\in\mathcal{X}} H_\epsilon(x) - \epsilon^\prime-\frac{\epsilon}{2} \geq  \alpha\cdot \OPT - \epsilon^\prime-\frac{\epsilon}{2}$$
	where the second inequality is from the assumption and the third inequality is because $H(x) \geq F(x)$ for all $x \in \cX$. Therefore, $\hat{x}$ is an $(\alpha,\epsilon^\prime+\epsilon/2)$-approximate solution to~\eqref{eq:DR-SFM-wasserstein}.
\Halmos
\end{proof}
Next we provide several structural properties of $H_\epsilon$ which allow us to use \cref{alg:conditional-gradient-upconcave-monotone,mirror-prox-non-differentiable} to solve \eqref{DR-SFM-H}. We defer the proofs to Appendix~\ref{proofs:robust}.
\begin{lemma}\label{lemma:H-structural}
Suppose \cref{ass:DR-SFM} holds. For any $\epsilon>0$, the function $H_\epsilon$ is nonnegative, monotone and up-concave. Moreover, $H_\epsilon$ is differentiable, and:
\begin{itemize}
	\item $H_\epsilon$ is $L_1$-Lipschitz, i.e., for any $x,y\in\mathcal{X}$, $|H_\epsilon(x)- H_\epsilon(y)| \leq L_1\|x-y\|.$
	
	\item $H_\epsilon$ is H\"older-smooth, i.e., for any $x,y\in\mathcal{X}$, 
	$$\|\grad H_\epsilon(x)-\grad H_\epsilon(y)\|_*\leq \lambda_1\|x-y\| + 2\lambda_2\theta\sqrt{\frac{L_1}{\epsilon}}\|x-y\|^{1/2}.$$
\end{itemize}
\end{lemma}

The next lemma states that we minimize $h_\epsilon(x,p)$ over $p \in \cP$ up to some small additive error $\Delta$, then we obtain an approximation of the gradient $\grad H_\epsilon(x)$.
\begin{lemma}\label{lemma:approximate-gradient}
	Let $x\in\mathcal{X}$, and let $p \in \cP$ be such that $h_\epsilon(x,p)\leq H_\epsilon(x) + \Delta$. Then 
	$$\left\|\grad H_\epsilon(x) -  \grad_{x}h_\epsilon(x,p) \right\|_*\leq \lambda_2 \theta\sqrt{\frac{2\Delta}{\epsilon}}.$$
\end{lemma}

\subsubsection{Convergence guarantees.}

We now have the necessary tools to apply \cref{alg:conditional-gradient-upconcave-monotone,mirror-prox-non-differentiable} to \cref{eq:DR-SFM-wasserstein} and derive their convergence rates. Note that \cref{lemma:H-structural} shows that $H_\epsilon$ is H\"{o}lder-smooth, i.e., it satisfies \cref{def:holder-smooth} for the function $h:\bbR_+ \to \bbR_+$ defined by
\[ h(r) = \lambda_1 r + 2 \lambda_2 \theta \sqrt{\frac{L_1}{\epsilon}} r^{1/2}. \]
\cref{cor:cg-convergence-rate}, we deduce the following convergence result for Algorithm~\ref{alg:conditional-gradient-upconcave-monotone} (noting that $\sigma=\min\{1,1/2\} = 1/2$ in the statement of \cref{cor:cg-convergence-rate}).
\begin{theorem}\label{thm:cg-dro}
	Suppose Assumptions \ref{ass:robust-SFM} and \ref{ass:DR-SFM} hold, and at each iteration $t$ of \cref{alg:conditional-gradient-upconcave-monotone}, we choose
	\begin{align*}
	p_t &= (\zeta^1,\ldots,\zeta^n) \in \cP \text{ s.t. } h_\epsilon((t-1)x_{t-1}/T,p_t) \leq H_\epsilon((t-1)x_{t-1}/T) + \Delta\\
	g_t &= \grad_{x} h_\epsilon((t-1)x_{t-1}/T, p_t) = \grad_{x} \sum_{i \in [n]} q_i \tilde{f}((t-1)x_{t-1}/T, \zeta^i).
	\end{align*}
	If $\Delta=O(\epsilon^3)$, then \cref{alg:conditional-gradient-upconcave-monotone} returns an $(1-1/e, \epsilon)$-approximate solution to \eqref{eq:DR-SFM-wasserstein} after $O(1/\epsilon^{2})$ iterations.
\end{theorem}

We may also apply \cref{mirror-prox-non-differentiable} to \cref{eq:DR-SFM-wasserstein}, which results in a worse approximation ratio, yet better dependence on $\epsilon$. \cref{lemma:H-structural} also shows that $H_\epsilon$ satisfies \cref{def:holder-smooth} for the constant function $h(r) = 2L_1$, i.e., it is Lipschitz-continuous. We deduce the following convergence result based on Corollary~\ref{cor:mp-convergence-rate} (noting that $\sigma=0$ in the statement of \cref{cor:cg-convergence-rate}).
\begin{theorem}\label{thm:mp-dro}
	Suppose Assumptions \ref{ass:robust-SFM} and \ref{ass:DR-SFM} hold, and at each iteration $t$ of \cref{alg:conditional-gradient-upconcave-monotone}, we choose
	\begin{align*}
		\tilde{p}_t &= (\tilde{\zeta}_t^1,\ldots,\tilde{\zeta}_t^n) \in \cP \text{ s.t. } h_\epsilon(v_t,p_t) \leq H_\epsilon(v_t) + \Delta\\
		\tilde{g}_t &= \grad_{x} h_\epsilon(v_t, \tilde{p}_t) = \sum_{i \in [n]} q_i \grad_{x} \tilde{f}(v_t, \tilde{\zeta}_t^i)\\
		p_t &= (\zeta_t^1,\ldots,\zeta_t^n) \in \cP \text{ s.t. } h_\epsilon(x_t,p_t) \leq H_\epsilon(x_t) + \Delta\\
		g_t &= \grad_{x} h_\epsilon(x_t, p_t) = \sum_{i \in [n]} q_i \grad_{x} \tilde{f}(x_t, \zeta_t^i).
	\end{align*}
	If $\Delta=O(\epsilon^3)$, %
	Algorithm~\ref{mirror-prox-non-differentiable} returns an $(1/2, \epsilon)$-approximate solution to \eqref{eq:DR-SFM-wasserstein} after $O(1/\epsilon^{4/3})$ iterations.
\end{theorem}

\section{Numerical experiments}

In this section, we present our experimental results to test the numerical performance of our algorithmic framework for the problem of maximizing a continuous DR-submodular function that is non-smooth or H\"older-smooth. We evaluated the efficacy of our algorithms, the continuous greedy method (Algorithm~\ref{alg:conditional-gradient-upconcave-monotone}) and the mirror-prox method (Algorithm~\ref{mirror-prox-non-differentiable}), on two sets of test instances: (1) the multi-resolution data summary problem with a non-differentiable utility function and (2) the distributionally robust movie recommendation problem. We used instances generated by synthetic simulated data for the multi-resolution data summary problem, which is described in Section~\ref{num:summary}. For the distributionally robust movie recommendation problem, we used the MovieLens 1M Dataset~\cite{movielens} to obtain the users' movie rating data. We explain the experimental setup for the multi-resolution data summary problem with a non-differentiable utility function in Section~\ref{num:summary} and that for the distributionally robust movie recommendation problem in Section~\ref{num:movie}. We report and summarize numerical results in Section~\ref{num:results}.

\subsection{Multi-resolution data summarization.}\label{num:summary}

Our first experiment is on the multi-resolution summary problem~\cite{bian} with a non-differentiable utility function. Given a collection of items $E=\{1,\ldots,k\}$, we assign each item $i$ a nonnegative score $x_i$ to measure its importance, by which we may recommend a subset of items. We set a threshold $\tau$ so that we report the set $S_{\tau}=\{i:x_{i}\geq\tau\}$ of items whose scores exceed the given threshold. By adjusting the value of $\tau$, we can decide the level of details or resolution of the summary. The scores of items basically represent the relative importance of items, and they are determined so that a utility function is maximized. The utility function is given by
\[
F(x) = \sum_{i \in [k]} \sum_{j \in [k]} \phi(x_{j})s_{ij}- \sum_{i \in [k]} \sum_{j \in [k]} x_{i}x_{j}s_{ij}
\]
where $s_{ij}\geq 0$ is the similarity index between two items $e_i$ and $e_j$. The first sum consists of the terms $\phi({x_j})s_{ij}$, which captures how much item $j$ contributes to item $i$ when $j$ has weight ${x_j}$. Here, $\Phi$ is a monotone concave function to model the diminishing returns property. The second sum consists of terms $x_ix_j s_{ij}$, which has a high value when two similar items $i$ and $j$ with a high similarity index $s_{ij}$ get large weights at the same time. Hence, taking away the second sum from the first encourages that if two items are similar, at most one of them gets a large score.

Note that $F$ is up-concave because the first sum is a concave function in $x$ and $-\sum_{i \in [k]} \sum_{j \in [k]} x_{i}x_{j}s_{ij}$ is DR-submodular with respect to $x$. We consider the case when $\phi$ is the following piece-wise linear function defined on $[0,1]$.
\[
\phi(x)=
\begin{cases}
7x & \text{if $x\in[0,\frac{1}{2}]$,} \\
6x+\frac{1}{2} & \text{if $x\in[\frac{1}{2},\frac{3}{4}]$,} \\
5x+\frac{5}{4} & \text{if $x\in[\frac{3}{4},1]$.}
\end{cases}
\]
In our experiments, we used randomly generated instances with $k=50$ items and similarity indices $s_{ij}$ sampled from the uniform distribution on $[0,1]$. For each instance, we determined a score vector by solving $\max_{x\in\mathcal{P}}F(x)$ where $\mathcal{P}=\{x\in[0,1]^k : \sum_{i\in [k]} x_{i}=5\}$. Note that the objective $F$ is not differentiable as $\phi$ is not differentiable. We tested the mirror-prox method, which works for non-differentiable objective functions, with the step size set to $\gamma_{t}=1/(2\sqrt{T})$. Although we do not have a theoretical ground for the continuous greedy method for the case of non-differentiable functions, we also ran the method to test its empirical effectiveness. We tested 30 randomly generated instances, and for each of the instances, we ran the algorithms for 30 iterations.

\subsection{Distributionally robust movie recommendation}\label{num:movie}

Our second problem is the distributionally robust formulation of the movie recommendation problem considered in~\cite{stan-submod,conditional}. We have a collection of $m$ movies $V = [m]$ that can be recommended to users.
Each user has a preference vector $\xi \in \Xi := \bbR^V$, where $\xi_j$ denotes the user's preference on movie $j \in V$. Then the evaluation of a set of movies $S$ presented to that user is given by the submodular set function $\tilde{f}_0(S,\xi)$ defined by
\[ \tilde{f}_0(S,\xi) := \max_{j \in S} \xi_j. \]
Let $P^*$ be the distribution of the preference index vectors of the entire population of the users. Then we can recommend a set of $k$ movies for the population of users based on solving 
$$\max_{|S|=k}\mathbb{E}_{\xi\sim P^*}\left[\tilde{f}_0(S,\xi)\right].$$
When $P^*$ is unknown to us, we may obtain a few sample preference vectors from some subset of $n$ users. Let $Q = \sum_{i \in [n]} q_i \delta_{\xi^i}$ be the empirical distribution of the collected preference vectors. Then we may obtain a list of movies to recommend by solving the following distributionally robust formulation:
$$\max_{|S|=k}\min_{P\in\mathcal{B}(Q,\theta)}\mathbb{E}_{\xi\sim P}\left[\tilde{f}_0(S,\xi)\right]$$
where $\mathcal{B}(Q,\theta)$ is the 2-Wasserstein ambiguity set of radius $\theta$ around the empirical distribution $Q$. We can solve the discrete problem by solving its continuous relaxation obtained by the multilinear extension and applying a proper rounding scheme, e.g.,~\cite{swap-rounding,pipage}, where the multilinear extension of $\tilde{f}_0(\cdot, \xi)$ is given by 
\[
\tilde{f}(x,\xi)=\sum_{S\subseteq[m]}\tilde{f}_0(S,\xi)\prod_{i\in S}x_{i}\prod_{i\notin S}(1-x_{i}).
\]
For our experiments, we used the MovieLens 1M Dataset from~\cite{movielens} that consists of $N=6041$ users and $m=4000$ movies. For each instance, we choose $n=10$ samples from the $N=6041$ data uniformly at random and consider the empirical distribution $Q$ on the samples. Then we consider the continuous relaxation
$$\max_{x\in \mathcal{X}}\min_{P\in\mathcal{B}(Q,\theta)}\mathbb{E}_{\xi\sim P}\left[\tilde{f}(x,\xi)\right]$$
where $\mathcal{X}=\{x\in[0,1]^m:\ \sum_{i=1}^m x_i = 5\}$. We reformulate this based on our framework described in Section~\ref{sec:dro}, so that we solve
\begin{align*}
\min_{x \in \cX} H_\epsilon(\vx), \quad H_\epsilon(\vx) &= \min_{p \in \cP} h_\epsilon(x,p)\\
h_\epsilon(x,p) &= \sum_{i \in [n]} q_i \left( \tilde{f}(x,\zeta^i) + \frac{\epsilon}{2 \theta^2} \|\zeta^i - \xi^i\|_2^2 \right)\\
p &= (\zeta^1,\ldots,\zeta^n) \in \cP = \left\{ (\zeta^1,\ldots,\zeta^n) \in \bbR^{m \times n} : \sum_{i \in [n]} q_i \|\xi^i - \zeta^i\|_2^2 \leq \theta^2 \right\}.
\end{align*}
We set the Wasserstein radius to $\theta=0.2$. We observed that each rating value $r_{ij}$ is from 1 to 5 in the dataset, in which case, Assumptions~\ref{ass:robust-SFM} and~\ref{ass:DR-SFM} are satisfied with $L_1=\lambda_1=L_2=5$ and $\lambda_2=10$. Lastly, to determine the accuracy parameter for inner minimization, we set $\epsilon=0.01$.

For each iteration of the continuous greedy algorithm (Algorithm~\ref{alg:conditional-gradient-upconcave-monotone}) and the mirror prox method (Algorithm~\ref{mirror-prox-non-differentiable}), we need to be able to compute the gradient $\grad_{x} \tilde{f}(x,\xi)$ of the multilinear extension. However, as there are exponentially many movies, computing the exact gradient is expensive. Instead, we obtained an unbiased estimator of the gradient $\grad_{x} \tilde{f}(x,\xi)$ based on the approach described in~{\cite[Section 9.3]{conditional}}.

Moreover, at each iteration, we need to solve the inner minimization problem. For that, we used gradient descent.
An issue with this approach is that it may be expensive to compute the values of $h_\epsilon,H_\epsilon,\grad_x h_\epsilon$, because they involve exponentially many multilinear terms as in the multilinear extension. To remedy this, we used the standard sampling approach to obtain an unbiased estimator for each quantity. To elaborate, note that $\tilde{f}(x,\xi) = \mathbb{E}_{S\sim x}\left[f(S,\xi)\right]$ where $S\sim x$ means that a set $S$ can be sampled by picking each movie $i\in[m]$ with probability $x_i$. Then obtaining sets $S_1,\ldots, S_B$ where $B=10$ and taking $(1/B)\sum_{b \in [B]} \tilde{f}_0(S_b,\xi)$, we obtain an unbiased estimator of $\tilde{f}(x,\xi)$. Similarly, we obtain an unbiased estimator of $\grad_{\xi} \tilde{f}(x,\xi)$.

For each problem instance, we ran both the conditional gradient descent algorithm and the mirror-prox algorithm for 300 iterations.

\subsection{Experimental results}\label{num:results}

\begin{figure*}[t]
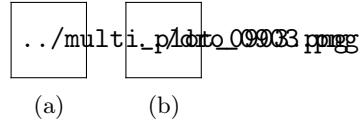

\centering
\subfigure[]
{
\includegraphics[scale=0.4]{../multi_plot_0903.png}
\label{fig:fig1}
}
\subfigure[]
{
\includegraphics[scale=0.4]{../dro_0903.png}
\label{fig:fig2}
}
\caption
{
Results of (left) the multi-resolution summary problem and (right) the distributionally robust movie recommendation.
}
\end{figure*}
Figure~\ref{fig:fig1} shows the numerical results for the multi-resolution summarization problem. We may observe that both the continuous greedy method and the mirror-prox method converge after a few iterations. However, the continuous greedy algorithm converges to a higher objective value than the mirror-prox method. This is interesting because the objective function for the multi-resolution summary problem is non-differentiable, in which case, the continuous greedy algorithm does not necessarily converge. A possible explanation would be that although the objective is non-differentiable, it is differentiable almost everywhere, and in particular, the function might have been differentiable at every iteration of the continuous greedy algorithm. The mirror-prox method converges to a value that is about 20\% less than the value to which the continuous greedy algorithm, and as our convergence result suggests, the value should be at least $1/2$ times the optimal value.

Figure~\ref{fig:fig2} shows the experimental results from the distributionally robust movie recommendation problem. We observe that the continuous greedy algorithm converges fast after less than 10 iterations. In contrast, the mirror-prox method exhibits a slower convergence pattern. Moreover, the continuous greedy algorithm converges to a higher value than the mirror-prox method, as expected from our theoretical results.

\begin{APPENDICES}
\section{Proofs from Section~\ref{sec:prelim}.}\label{app:proofs}
\begin{proof}{Proof of \cref{prop:reformulation-1}.}
	As $\mathcal{B}(Q,\theta)$ is a $2$-Wasserstein ball, the strong duality result of \citet[Eq. (11)]{BlanchetMurthy2019} states that
	\begin{align}\label{DR-SFM-G'}
		\begin{aligned}
			F(x)&= -\sup_{P \in \mathcal{B}(Q,\theta)} \mathbb{E}_{\xi \sim P} \left[ -\tilde{f}(x,\xi) \right]\\
			&= - \inf_{\lambda \geq 0} \left\{ \theta^2 \lambda + \mathbb{E}_{\xi \sim Q} \left[ \sup_{\zeta\in\Xi} \left\{ -f(x,\zeta) - \lambda \|\xi - \zeta\|_2^2 \right\} \right] \right\}\\
			&= \sup_{\lambda \geq 0} \left\{ \mathbb{E}_{\xi \sim Q} \left[ \inf_{\zeta\in\Xi} \left\{ \tilde{f}(x,\zeta) + \lambda \|\xi - \zeta\|_2^2 \right\} \right] - \theta^2 \lambda \right\}.
		\end{aligned}
	\end{align}
	Furthermore, by our assumption that $Q=\sum_{i\in[n]}q_i\delta_{\xi^i}$ has a finite support, the last term of~\eqref{DR-SFM-G'} can be rewritten as the following  finite-dimensional max-min problem:
	\begin{equation}\label{DR-SFM-G-1}
		F(x)=\sup_{\lambda \geq 0} \sum_{i \in [n]} q_i \left\{\inf_{\zeta\in \Xi} \left\{\tilde{f}(x,\zeta) + \lambda \|\xi^i - \zeta\|_2^2 \right\}  - \theta^2 \lambda\right\} .
	\end{equation}

	As $\tilde{f}(x,\zeta)$ is convex in $\zeta$, the function $\sum_{i \in [n]} q_i \tilde{f}(x,\zeta^i)$ is also convex. Moreover,  $\sum_{i\in[n]} q_i\|\xi^i - \zeta^i\|_2^2  - \theta^2\leq0$ is a convex constraint with respect to $\zeta^1,\ldots,\zeta^n$. Notice that choosing $\zeta^i=\xi^i$ for $i\in[n]$ gives $\sum_{i\in[n]} q_i\|\xi^i - \zeta^i\|_2^2  - \theta^2=-\theta^2<0$. Then strong Lagrangian duality holds to give
	\begin{align*}
		&\inf_{(\zeta^1,\cdots,\zeta^n)\in\mathrm{Z}}\sum_{i\in[n]}q_i \tilde{f}(x,\zeta^i)\\
		&= \sup_{\lambda\geq 0}\inf_{\zeta^1,\ldots,\zeta^n\in\Xi}\sum_{i\in[n]}q_i f(x,\zeta^i)-\lambda\left(\sum_{i\in[n]} q_i\|\xi^i - \zeta^i\|_2^2  - \theta^2\right),
	\end{align*}
	and it follows from~\eqref{DR-SFM-G-1} that
	$$\inf_{(\zeta^1,\cdots,\zeta^n)\in\mathrm{Z}}\sum_{i\in[n]}q_i \tilde{f}(x,\zeta^i)=F(x)=\inf_{P \in \mathcal{B}(Q,\theta)} \mathbb{E}_{\xi \sim P} \left[ \tilde{f}(x,\xi) \right],$$
	as required.
	\Halmos
\end{proof}

\begin{proof}{Proof of Lemma~\ref{diff-differential}.}
Take $x\in\mathbb{R}^d$ and $g\in \partial^\uparrow F(x)$. Let $j\in[d]$, and let $e_j$ be the $j$th standard basis vector. Since $F$ is up-concave, we have $F(x + t e_j)- F(x) \leq  g^\top(t e_j)$ for any $t\in\mathbb{R}_+$.
Then it follows that
$$e_j^\top \grad F(x)=\lim_{t\to 0+}\frac{F(x + t e_j)- F(x)}{t} \leq  e_j^\top g,$$
and therefore $e_j^\top (\grad F(x)-g)\leq 0$. Similarly, we have $F(x - t e_j)- F(x) \leq  g^\top(-t e_j)$ for any $t\in\mathbb{R}_+$, which implies that
$$-e_j^\top \grad F(x)=\lim_{t\to 0+}\frac{F(x - t e_j)- F(x)}{t} \leq  -e_j^\top g.$$
Thus we obtain $e_j^\top (\grad F(x)-g)\geq 0$. We have just proved that $e_j^\top (\grad F(x)-g)= 0$ for every $j\in[d]$, implying in turn that $g=\grad F(x)$. Therefore, $\partial^\uparrow F(x) = \left\{\grad F(x)\right\}$.
\Halmos
\end{proof}

\section{Proofs from Section~\ref{sec:robust}.}\label{proofs:robust}

\begin{proof}{Proof of \cref{lemma:robust-SFM-objective}.}
	It is trivial to see that $F$ is non-negative and up-concave (which follows since the minimum of arbitrary concave functions is concave). To show that $F$ is monotone, let us take $x,y\in\mathbb{R}^d$ such that $x\leq y$. For each $p \in \cP$, we have $f(x,p)\leq f(y,p)$. Taking the minimum of the left-hand side of this inequality over $p \in \cP$, it follows that
	$F(x) \leq f(y,p)$ for each $p \in \cP$. Then taking the minimum of the right-hand side, we obtain the desired $F(x)\leq F(y)$.

	For fixed $x$, let $p_{x}^* \in \argmin_{p \in \cP} f(x,p)$. Since $f(x,p_{x}^*)$ is differentiable in $x$, $\{ \grad f(x,p_{x}^*) \} = \partial^\uparrow f(x,p_{x}^*)$. Then for any $y \in \bbR^d$:
	$$F(y)\leq f(y,p_{x}^*) \leq f(x,p_{x}^*) + \grad f(x,p_{x}^*)^\top(y-x)=F(x) + \grad f(x,p_{x}^*)^\top(y-x),$$
	implying in turn that $\grad f(x,p_{x}^*) \in \partial^\uparrow F(x)$.
	
	Note that $\|\grad f(x, p)\|_* \leq L$ follows since each $f(x,p)$ is differentiable and $L$-Lipschitz continuous. Finally, let us argue that $F$ is $L$-Lipschitz continuous with respect to the norm $\|\cdot\|$. For $x,y\in\mathbb{R}^d$ and $p \in \cP$, we have
	$$f(x,p) = f(x,p) - f(y,p) + f(y,p) \leq L\|x-y\| + f(y,p).$$
	Therefore, we have $f(x,p) \leq L\|x-y\| + f(y,p)$. Taking the minimum of the left-hand side over $p \in \cP$, we deduce that $F(x) \leq L\|x-y\| + f(y,p)$. Then we take the minimum of its right-hand side over $p \in \cP$, which results in $F(x)\leq L\|x-y\|+ F(y)$. Switching $x$ and $y$, we have $F(y)\leq L\|y-x\|+ F(x)$. Therefore, it follows that $|F(x)-F(y)|\leq L\|x-y\|$.
	\Halmos
\end{proof}

To prove \cref{lemma:H-structural,lemma:approximate-gradient} we need the following result, which is based on standard convexity arguments.
\begin{lemma}\label{lemma:strong-convexity}
Given $x \in \bbR_+^d$ and $\epsilon > 0$, denote $p_x^* \in \argmin_{p \in \cP} h_\epsilon(x,p)$ by $p_x^* = (\zeta_x^1,\ldots,\zeta_x^n)$. Then for any $p = (\zeta^1,\ldots,\zeta^n) \in \cP$, we have
\[ h_\epsilon(x,p) - h_\epsilon(x,p_{x,\epsilon}^*) \geq \frac{\epsilon}{2 \theta^2} \sum_{i \in [n]} q_i \|\zeta^i - \zeta_x^i\|_2^2. \]
\end{lemma}
\begin{proof}{Proof.}
Recall that $\tilde{f}(x,\zeta)$ is convex in $\zeta$, and that $h_\epsilon(x,p) = f(x,p) + \frac{\epsilon}{2 \theta^2} \sum_{i \in [n]} q_i \|\xi^i - \zeta^i\|_2^2 = \sum_{i \in [n]} q_i \left( \tilde{f}(x,\zeta^i) + \frac{\epsilon}{2 \theta^2} \|\xi^i - \zeta^i\|_2^2 \right)$. Let $g_x^i$ be a subgradient of $\tilde{f}(x,\cdot)$ at $\zeta_x^i$. Then, by strong convexity, for each $i \in [n]$ we have
\[ \tilde{f}(x,\zeta^i) + \frac{\epsilon}{2 \theta^2} \|\xi^i - \zeta^i\|_2^2 \geq \tilde{f}(x,\zeta_x^i) + \frac{\epsilon}{2 \theta^2} \|\xi^i - \zeta_x^i\|_2^2 + \left(g_x^i + \frac{\epsilon}{\theta^2} (\zeta_x^i - \xi^i) \right)^\top (\zeta^i - \zeta_x^i) + \frac{\epsilon}{2\theta^2} \|\zeta^i - \zeta_x^i\|_2^2. \]
Denote $g_x = (q_1 g_x^1, \ldots, q_n g_x^n)$, which is a subgradient of $f(x,\cdot)$ at $p_x^*$. Multiplying both sides by $q_i$ then summing over $i \in [n]$ gives
\[ h_\epsilon(x,p) \geq h_\epsilon(x,p_x^*) + g_x^\top (p - p_x^*) + \frac{\epsilon}{\theta^2} \sum_{i \in [n]} (\xi^i - \zeta_x^i)^\top (\zeta^i - \zeta_x^i) + \frac{\epsilon}{2\theta^2} \sum_{i \in [n]} q_i \|\zeta^i - \zeta_x^i\|_2^2. \]
Note that
\[ g_x^\top (p - p_x^*) + \frac{\epsilon}{\theta^2} \sum_{i \in [n]} (\xi^i - \zeta_x^i)^\top (\zeta^i - \zeta_x^i) \geq 0 \]
for any $p=(\zeta^1,\ldots,\zeta^n) \in \cP$ by optimality of $p_x^*$, so we have
\[ h_\epsilon(x,p) \geq h_\epsilon(x,p_x^*) + \frac{\epsilon}{2\theta^2} \sum_{i \in [n]} q_i \|\zeta^i - \zeta_x^i\|_2^2 \]
as required.
\Halmos
\end{proof}

\begin{proof}{Proof of \cref{lemma:H-structural}.}
	Nonnegativity, monotonicity and up-concavity of $H_\epsilon$ follows since $h_\epsilon(x, p)$ satisfies the required parts of \cref{ass:robust-SFM}, and applying the proof from \cref{lemma:robust-SFM-objective}.
	
	Note that any minimizer $p^*_{x,\epsilon}$ for $h_\epsilon(x,p)$ is unique, as $h_\epsilon$ is (by construction) strongly convex in $p$. Therefore $H_\epsilon(x)$ is differentiable by the envelope theorem, and
	\[ \grad H_\epsilon(x) = \grad_{x} h_\epsilon(x, p_{x,\epsilon}^*) = \grad_{x} f(x, p_{x,\epsilon}^*). \]

	We now show that $H_\epsilon$ is $L_1$-Lipschitz continuous. For each $(\zeta^1,\cdots,\zeta^n)\in \cP$, we know that $\zeta^i\in\Gamma$ for each $i\in[n]$. For $x,y\in\mathcal{X}$ and $\zeta^i\in\Gamma$, we have
	\begin{align*}
		\tilde{f}(x,\zeta^i)=\tilde{f}(x,\zeta^i)-\tilde{f}(y,\zeta^i)+\tilde{f}(y,\zeta^i)\leq L_1\lVert x-y\rVert + \tilde{f}(y,\zeta^i)
	\end{align*}
	where the inequality holds due to Assumption~\ref{ass:DR-SFM}. Since $\sum_{i\in[n]}q_i=1$,
	\[
	\sum_{i\in[n]}q_{i}\tilde{f}(x,\zeta^i)\leq L_1\lVert x-y\rVert+\sum_{i\in[n]}p_{i}\tilde{f}(y,\zeta^i).
	\]
	By adding the term $\frac{\epsilon}{2\theta^2}\sum_{i\in[n]}q_i\left\|\xi^i-\zeta^i\right\|_2^2$ to both sides, we obtain
	\begin{align*}
		h_\epsilon(x,p) \leq L_1\lVert x-y\rVert + h_\epsilon(y,p).
	\end{align*}
	By taking the infimum of its left-hand side over $p\in \cP$ and that of the right-hand side next, we obtain
	\begin{align*}
		H_\epsilon(x) &\leq L_1\lVert x-y\rVert + H_\epsilon(y),
	\end{align*}
	and therefore, $H(x)-H(y)\leq L_1\lVert x-y\rVert$. Similarly, we deduce that $-H(x)+H(y)\leq L_1\lVert x-y\rVert$, hence $H$ is $L_1$-Lipschitz continuous in the norm $\|\cdot\|$.
	
	We finally show that $H_\epsilon(x)$ is H\"older-smooth. \cref{lemma:strong-convexity} gives
	\[ h_\epsilon(x,p) - h_\epsilon(x,p_{x,\epsilon}^*) \geq \frac{\epsilon}{2 \theta^2} \sum_{i \in [n]} q_i \|\zeta^i - \zeta_x^i\|_2^2 \]
	for any $p = (\zeta^1,\ldots,\zeta^n) \in \cP$, where $p_{x,\epsilon}^* = \left( \zeta_{x,\epsilon}^1,\ldots,\zeta_{x,\epsilon}^n \right) \in \argmin_{p' \in \cP} h_\epsilon(x,p')$.
	Let $\hat{x},\tilde{x}$ be two arbitrary points, and let $\hat{p}=(\hat{\zeta}^1,\ldots,\hat{\zeta}^n),\tilde{p}=(\tilde{\zeta}^1,\ldots,\tilde{\zeta}^n)$ be the minimizers of $h_\epsilon(\hat{x},p)$ and $h_\epsilon(\tilde{x},p)$ over $p \in \cP$ respectively.
	Note that
	\begin{subequations}\label{eq:lemma:H-smooth-1}
		\begin{align}
			\|\grad H_\epsilon(\hat{x})-\grad H_\epsilon(\tilde{x})\|_*&= \|\grad_{x} h_\epsilon(\hat{x},\hat{p})-\grad_{x} h_\epsilon(\tilde{x},\tilde{p})\|_*\\
			&\leq \|\grad_{x} h_\epsilon(\hat{x},\hat{p})-\grad_{x} h_\epsilon(\tilde{x},\hat{p})\|_* +\|\grad_{x} h_\epsilon(\tilde{x},\hat{p})-\grad_{x} h_\epsilon(\tilde{x},\tilde{p})\|_*\\
			&\leq \lambda_1\|\hat{x}-\tilde{x}\| + \lambda_2 \sum_{i\in[n]}q_i\|\hat{\zeta}^i-\tilde{\zeta}^i\|_2
		\end{align}
	\end{subequations}
	where the second inequality is due to Assumption~\ref{ass:DR-SFM}. Next, to bound $\|\grad H(\hat{x})-\grad H(\tilde{x})\|_*$ based on~\eqref{eq:lemma:H-smooth-1}, we consider the term $\sum_{i\in[n]}q_i\|\hat{\zeta}^i-\tilde{\zeta}^i\|_2$.
	By the Cauchy-Schwarz inequality, we obtain the following:
	\begin{align*}
		\sum_{i\in[n]}q_i\|\hat{\zeta}^i-\tilde{\zeta}^i\|_2&\leq \left( \sum_{i\in[n]}q_i \right)^{1/2} \left( \sum_{i\in[n]}q_i\|\hat{\zeta}^i-\tilde{\zeta}^i\|_2^2 \right)^{1/2}\\
		&\leq \sqrt{\frac{2\theta^2}{\epsilon}}(h_\epsilon(\tilde{x},\hat{p})-h_\epsilon(\tilde{x},\tilde{p}))^{1/2}.
	\end{align*}
	Then it follows that
	\begin{align*}
		h_\epsilon(\tilde{x},\hat{p})-h_\epsilon(\tilde{x},\tilde{p}) &\leq |h_\epsilon(\tilde{x},\hat{p})-h_\epsilon(\hat{x},\hat{p})|+|H_\epsilon(\hat{x})- H_\epsilon(\tilde{x})|\\
		&\leq 2L_1\|\hat{x}-\tilde{x}\|
	\end{align*}
	holds because of Assumption~\ref{ass:DR-SFM} and $L_1$-Lipschitz continuity of $H$. This then implies
	$$\|\grad H_\epsilon(\hat{x})-\grad H_\epsilon(\tilde{x})\|_*\leq \lambda_1\|\hat{x}-\tilde{x}\| + 2\lambda_2\theta\sqrt{\frac{L_1}{\epsilon}}\|\hat{x}-\tilde{x}\|^{1/2},$$
	as required.
	\Halmos
\end{proof}

\begin{proof}{Proof of Lemma~\ref{lemma:approximate-gradient}.}
	Let $p_x^* = (\zeta_x^1,\ldots,\zeta_x^n) \in \cP$ be such that $H_\epsilon(x)=h_\epsilon(x,p_x^*)$. Then
	\begin{equation}\label{eq:lemma:approximate-gradient-3}
		h_\epsilon(x,p)-h_\epsilon(x,p_x^*)=h_\epsilon(x,p)-H_\epsilon(x)\leq\delta.
	\end{equation}
	\cref{lemma:strong-convexity} shows that for any $p=(\zeta^1,\ldots,\zeta^n) \in \cP$ we have
	\begin{equation}\label{eq:lemma:approximate-gradient-4}
		\sum_{i \in [n]} q_i \|\zeta^i - \zeta_x^i\|_2^2 \leq \frac{2\theta^2}{\epsilon} (h_\epsilon(x,p)-H_\epsilon(x)) \leq \frac{2\theta^2\delta}{\epsilon}.
	\end{equation}
	Next, based on the last statement of \cref{ass:DR-SFM}, we deduce the following.
	\begin{align*}
		\left\|\grad H_\epsilon(x) -  \grad_{x}h_\epsilon(x,p)\right\|_*&\leq \sum_{i\in[n]}q_i\left\|\grad_{x}\tilde{f}(x,\zeta_x^i)-\grad_{x}\tilde{f}(x,\zeta^i)\right\|_*\\
		&\leq \lambda_2 \sum_{i\in[n]}q_i\|\zeta^i-\zeta_x^i\|_2\\
		&\leq \lambda_2 \left(\sum_{i\in[n]}q_i \right)^{1/2}\cdot \left( \sum_{i\in[n]}q_i\|\zeta^i-\zeta_x^i\|_2^2 \right)^{1/2}\\
		&\leq \lambda_2\theta\sqrt{\frac{2\delta}{\epsilon}}
	\end{align*}
	where the first inequality is by the triangle inequality, the second inequality follows from the last statement of \cref{ass:DR-SFM}, the third inequality is due to the Cauchy-Schwarz inequality, and the last inequality comes from~\eqref{eq:lemma:approximate-gradient-4}.
	\Halmos
\end{proof}

\section{Discrete submodular functions.}\label{dro-discrete}

\begin{lemma}\label{lemma:assumptions}
Let $V$ be a finite ground set, and $\tilde{f}_0:2^V\times\Gamma\to\mathbb{R}_+$ be a nonnegative function such that $\tilde{f}_0(S,\xi)$ is monotone and submodular with respect to $S\subseteq V$, and convex and $L_2$-Lipschitz continuous with respect to $\xi\in \Gamma$. We further assume that $\tilde{f}_0(\emptyset,\xi)=0$ for $\xi\in \Gamma$. If $\tilde{f}(x,\xi)$ is the multilinear extension of $\tilde{f}_0(S,\xi)$ for each $\xi\in\Gamma$, then last three statements in \cref{ass:DR-SFM} hold for $\tilde{f}$ with $L_1=\lambda_1=\sup_{(j,\xi)\in V\times \Gamma}\tilde{f}_0(\{j\},\xi)$ and $\lambda_2= 2L_2$.
\end{lemma}
\begin{proof}{Proof.}
By~\citet{matroidcon}, we have for $y\in [0,1]^V$ and $\zeta\in \Gamma$,
$$0\leq\frac{\partial \tilde{f}(y,\zeta)}{\partial x_i}=\tilde{f}(y,\zeta)|_{y_i=1} -\tilde{f}(y,\zeta)|_{y_i=0}\leq \tilde{f}_0(\{i\},\zeta)$$
where the last inequality is by submodularity of $\tilde{f}_0$. Therefore, $\|\grad_{x} \tilde{f}(y,\zeta)\|_\infty\leq \max_{i\in V}\tilde{f}_0(\{i\},\zeta)$, which implies that, for every $\xi\in\Gamma$, $\tilde{f}(x,\xi)$ is $L_1$-Lipschitz continuous in the $\ell_1$-norm over $x \in \mathcal{X}$ where $L_1=\sup_{(j,\xi)\in V\times \Gamma}\tilde{f}_0(\{j\},\xi)$. Thus the fourth statement of \cref{ass:DR-SFM} holds

Next we argue that the fifth statement of \cref{ass:DR-SFM} holds. By \citet[Lemma C.1]{grad-sfm}, for $y\in [0,1]^V$ and $\zeta\in \Gamma$, 
$$- \max_{j\in V} \tilde{f}_0(\{j\},\zeta)\leq\frac{\partial^2 \tilde{f}(y,\zeta)}{\partial x_i\partial x_j} \leq0$$
Let $z\in\mathbb{R}^V$. Note that $$\left|z^\top \grad^2_{x}\tilde{f}(y,\zeta) z\right|\leq \sum_{i,j\in V} \left|\frac{\partial^2 \tilde{f}(y,\zeta)}{\partial x_i\partial x_j}\right|\cdot \left|z_iz_j\right|\leq \max_{j\in V} \tilde{f}_0(\{j\},\zeta)\|z\|_1^2.$$
Therefore, $\tilde{f}(\cdot,\xi)$ for a fixed $\xi$ is $\left(\max_{j\in V} \tilde{f}_0(\{j\},\xi)\right)$-smooth, so for any $y^1,y^2\in[0,1]^V$, we have
$$\|\grad_{x}\tilde{f}(y^1,\zeta)-\grad_{x}\tilde{f}(y^2,\zeta)\|_\infty\leq \sup_{(j,\xi)\in V\times\Gamma}\tilde{f}_0(\{j\},\xi)\|y^1-y^2\|_1.$$
Therefore, the fifth statement of \cref{ass:DR-SFM} holds.

We now consider the sixth statement of \cref{ass:DR-SFM}. It is known~\cite{matroidcon} that
$$\frac{\partial}{\partial x_j}\tilde{f}(y,\zeta)=\sum_{S\subseteq V\setminus\{j\}} \left(\tilde{f}_0(S\cup \{j\},\zeta) - \tilde{f}_0(S,\zeta)\right) \prod_{\ell\in S}y_j\prod_{\ell \in V\setminus\{j\}}(1-y_j).$$
Then it follows that
\begin{align}
\begin{aligned}
&\frac{\partial}{\partial x_j}\tilde{f}(y,\zeta^1)-\frac{\partial}{\partial x_j}\tilde{f}(y,\zeta^2)\\
&= \sum_{S\subseteq V\setminus\{j\}} \left(\left(\tilde{f}_0(S\cup \{j\},\zeta^1) - \tilde{f}_0(S,\zeta^1)\right)-\left(\tilde{f}_0(S\cup \{j\},\zeta^2) - \tilde{f}_0(S,\zeta^2)\right)\right) \prod_{\ell\in S}y_j\prod_{\ell \in V\setminus\{j\}}(1-y_j)\\
&=\sum_{S\subseteq V\setminus\{j\}} \left(\left(\tilde{f}_0(S\cup \{j\},\zeta^1) -\tilde{f}_0(S\cup \{j\},\zeta^2) \right)-\left( \tilde{f}_0(S,\zeta^1)- \tilde{f}_0(S,\zeta^2)\right)\right) \prod_{\ell\in S}y_j\prod_{\ell \in V\setminus\{j\}}(1-y_j)\\
&\leq \sum_{S\subseteq V\setminus\{j\}} \left(L_2\|\zeta^1-\zeta^2\|+L_2\|\zeta^1-\zeta^2\|\right) \prod_{\ell\in S}y_j\prod_{\ell \in V\setminus\{j\}}(1-y_j)\\
&= 2L_2\|\zeta^1-\zeta^2\|\sum_{S\subseteq V\setminus\{j\}}  \prod_{\ell\in S}y_j\prod_{\ell \in V\setminus\{j\}}(1-y_j)\\
&= 2L_2\|\zeta^1-\zeta^2\|
\end{aligned}
\end{align}
where the inequality holds because $\tilde{f}_0(S,\cdot)$ is $L_2$-Lipschitz continuous for any $S\subseteq V$. Hence the sixth statement of \cref{ass:DR-SFM} holds true with $\lambda_2 =2L_2$.
\Halmos
\end{proof}

\end{APPENDICES}

 \section*{Acknowledgments.}
This research is supported, in part, by KAIST Starting Fund (KAIST-G04220016), FOUR Brain Korea 21 Program (NRF-5199990113928), and National Research Foundation of Korea (NRF-2022M3J6A1063021).

\bibliographystyle{informs2014} %
\bibliography{mybibfile} %

\end{document}